\documentclass[12pt] {article}
\usepackage{amsmath, amsfonts, amsthm, amssymb, color, hyperref, extarrows, enumitem}

\newtheorem{theorem}{Theorem}[section]

\newtheorem{definition}[theorem]{Definition}
\newtheorem{cor}[theorem]{Corollary}
\newtheorem{lem}[theorem]{Lemma}
\newtheorem{pro}[theorem]{Proposition}

\numberwithin{equation}{section}

\usepackage[T1]{fontenc}
\DeclareMathOperator{\Hi}{O}
 
{\theoremstyle{definition}
\newtheorem{example}[theorem]{Example} 
\newtheorem{remark}[theorem]{Remark}}

\usepackage[titletoc]{appendix}

\usepackage{cite}
 
\hypersetup{
	colorlinks   = true,
	citecolor    = magenta}

	\usepackage[titletoc]{appendix}

 \usepackage{tikz}
 
\usepackage{capt-of}
 
\begin{document}

\title{\vspace{-1.2cm} \bf Bergman representative coordinate, constant holomorphic curvature and a multidimensional generalization of Carath\'eodory's theorem 
 \rm}  
\author{Robert Xin Dong \quad  and  \quad  Bun Wong}

\date{} 
\maketitle

\begin{abstract}

By using the Bergman representative coordinate and Calabi’s diastasis, we extend a theorem of Lu to bounded pseudoconvex domains whose Bergman metric is incomplete with constant holomorphic sectional curvature. 
We characterize such domains that are biholomorphic to a ball possibly less a relatively closed pluripolar set. 
We also provide a multidimensional generalization of Carath\'eodory's theorem on the continuous extension of the biholomorphisms up to the closures.
In particular,
 sufficient conditions are given, in terms of the Bergman kernel,  for the boundary of a biholomorphic ball to be a topological sphere.

\end{abstract}

\renewcommand{\thefootnote}{\fnsymbol{footnote}}
\footnotetext{\hspace*{-7mm} 
\begin{tabular}{@{}r@{}p{16.5cm}@{}}
& Keywords. Bergman kernel, Bergman metric, Bergman representative coordinate, biholomorphic map, Carath\'eodory's theorem, holomorphic sectional curvature, pluripolar set\\
& Mathematics Subject Classification. Primary 32F45; Secondary 32H10, 32T05, 32D20\\

\end{tabular}}

\section{Introduction}

For a bounded simply-connected domain  $D$  in the complex plane $\mathbb C$,  let  $S: \mathbb D \to D$ be the unique biholomorphism with $S(0) = 0 \in D$ and $S'(0) > 0$, where $\mathbb D$ denotes the unit  disc around the origin. It is well-known that $S$ extends continuously up to $\overline {\mathbb D}$ if and only if $\partial D$ is locally connected, and a celebrated theorem of Carath\'eodory \cite{Cara} states that $S$ extends to a homeomorphism of   $\overline{\mathbb D}$ onto $ \overline D$  if and only if  $\partial D$ is a Jordan curve. In the latter case, if the Jordan curve is $C^\infty$-smooth, then $S$ and all its derivatives extend continuously to   $ \overline {\mathbb D}$ (see \cite{BK}).
  Bergman discovered that the Riemann map associated to  $D $  can be expressed simply in terms of his kernel function $K(z, t)$. (see \cite{Bell06} and \cite[Chap. VI]{Ber70}).
  Let $f$ be  a biholomorphic map  from $D$ onto  $\mathbb D$  such that $f(p) =0$   for some point $p \in D$.
  Then, $K(z, t) =  f^{\prime}(z) \overline{f^{\prime}(t)} \pi^{-1} (1-f(z)\overline{f(t)})^{-2} $ and
     $$  
  K^{-1}(z, p)  \left. \frac{\partial}{\partial \overline {t } } \right|_{t=p} K(z, t)    =   {2 f^{\prime}(p)}  f(z) + \overline{ \frac{ {f^{\prime\prime}(p)}}{ {f^{\prime }(p)}}}.
   $$ 
 Bergman in \cite{Ber} introduced his representative coordinates  as a tool   of generalizing the Riemann mapping theorem to $ \mathbb C^n, n \geq 1$.
      Let  $\Omega \subset \mathbb C^n$  be a bounded domain whose  Bergman metric is denoted by $g$.
Relative to a point $p\in  \Omega$, the Bergman representative coordinate $T (z) =(w_1, ... , w_n)$ is defined as
\begin{equation} \label{rep}
w_{\alpha} (z):=\sum _{j=1}^{n} g^{\bar j \alpha }(p) \left(K^{-1}(z, p)  \left. \frac{\partial}{\partial \overline {t_j} } \right|_{t=p} K(z, t)  -    \left. \frac{\partial}{\partial \overline {t_j} } \right|_{t=p}  \log K(t, t)\right), \quad \alpha = 1, \dotsc, n,
\end{equation}
where  
$K(z , t)$ is the Bergman kernel and $g^{\bar j \alpha }$ are the entries of the inverse of the matrix $[g_{\alpha \bar{j}}]$ associated with the Bergman metric. 
Since the possible obstruction in   \eqref{rep}   is that $K(z , p)$ may have zero,  $T (z)$ is only generally defined and holomorphic outside the zero set of $K(z , p)$. Studying zeros of the Bergman kernel attracted  much interest, and domains for which the Bergman kernel is zero-free are known as Lu Qi-Keng domains (or those satisfying the Lu Qi-Keng conjecture, cf. \cite{Bo}), after  Lu's well-known paper \cite{Lu} on his  uniformization theorem.   In that paper, Lu proved that  for a bounded domain in $\mathbb C^n$ with a complete Bergman metric of constant holomorphic sectional curvature, the Bergman representative coordinate  is a  {\it biholomorphism that  maps  $\Omega$  to a Euclidean ball}. 
   Alternative proofs of  Lu's theorem are available  by following  Bergman's key idea that biholomorphic mappings become linear when represented in his   coordinates,  
   cf. \cite{GKK, Yoo};
see also \cite{CW} for a simplification of Lu's proof by Cheung and the second author.
     Lu's theorem also played a decisive role in    the resolution of the Cheng conjecture, which asserts that for a smoothly bounded strictly pseudoconvex domain, the Bergman metric is K\"{a}hler-Einstein if and only if the domain is biholomorphic to a ball, and was recently confirmed by Huang and Xiao \cite{HX1, HX}, after  the previous works of Fu and the second author \cite{FW} and Nemirovski and Shafikov \cite{NS}. See also the related work of Li \cite{Li}.
  In all the above proofs,  the completeness of the Bergman metric was crucial.

        \medskip

       \medskip

We aim to extend Lu's theorem towards the Bergman-incomplete situation. 
One motivation comes from the remarkable progresses in the past decade around the sharp $L^2$ extension theorems of the Ohsawa-Takegoshi type (see \cite{OT, Bl13, GZ, BL}, the survey papers \cite{Bl, O20, O20B, Z}, and the references therein), which particularly solved a long-standing conjecture of
 Suita \cite{Su}. Its equality part  characterizes   Riemann surfaces that are 
 {\it biholomorphic to a disc possibly less a relatively closed polar set.}
Here,  a polar set is the local singularity set of a subharmonic function. 
In fact, the biholomorphism can be expressed in terms of the  Bergman representative coordinate (see \cite[Chap. 4]{O18}) as the possible polar part is negligible for $L^2$ holomorphic 1-forms.
On the other hand, we use  Bergman's coordinates to  generalize Carath\'eodory's theorem by studying the continuous extension of biholomorphism up to the closures of higher dimensional  pseudoconvex domains; this particularly generates multidimensional  results analogous to the equality  part of Suita's conjecture.
Our study  mostly focuses on domains with worse boundaries, since a bounded domain which is complete with respect to the Bergman metric
 is pseudoconvex (see  \cite{Bre}), and 
     the Bergman completeness holds for  any bounded pseudoconvex  domain with H\"older boundary in $\mathbb C^n$ (see  \cite{AHP, Chen}).  
  By using ideas from Lu’s original paper as well as Calabi’s concept of diastasis, the authors in \cite{DWo} substituted the Bergman completeness with other conditions before dropping it completely.  There, the authors showed that a bounded pseudoconvex  domain $\Omega \subset \mathbb C^n$ whose Bergman metric has negative constant holomorphic sectional curvature is  a Lu Qi-Keng domain, and such $\Omega$ must be {\it  biholomorphic to a ball possibly less a relatively closed pluripolar  set} if there exists some point $p \in \Omega$ such that 
  \begin{enumerate}
 \item [$1.$]    $|K(z, p)|$ is bounded from above by a finite constant $\mathcal C>0$ for any $z \in \Omega$, and   
 \item [$2.$]    the Bergman representative coordinate $T $ relative to $p$ is continuous up to $ \overline\Omega$.
 \end{enumerate}

         \medskip
         
  In this paper, we are able to drop the above second condition for planar domains. To state our results on  the continuous extension of biholomorphism up to the closures, we consider the following topological condition.

 \begin{definition} \label{def-local connect}
  Let $\Omega$ be a bounded domain in  $ \mathbb C^n, n \ge 1$.
We say $\Omega$   satisfies  local $C^1$-connectivity around $q \in \partial \Omega$, if 
 for any $\epsilon >0$, 
there exists an open set $U_{\epsilon} \ni q$ such that for any $x, y \in (U_\epsilon \cap \Omega)$, there exists a piece-wise $C^1$  path $\gamma_{x, y}  \subset \Omega$  joining $x$ and $y$ and the length of $\gamma_{x,y}$ is less than $\epsilon$.

  \end{definition}

          \medskip

Examples of domains satisfying the  local $C^1$-connectivity include  convex domains,  strictly pseudoconvex domains, domains with piecewise $C^1$-smooth boundaries, and 
domains that are locally $C^1$-diffeomorphic to a convex domain, etc.
 See Remark \ref{Example C^1 conn}.   
 
       \medskip

\begin{theorem}  \label{1-dim}   Let   $\Omega \subset \mathbb C $ be a bounded domain whose Bergman metric  has  Gaussian curvature identically equal to $- 2$. If there exists some point   $p \in \Omega$ such that 
     $
  \left|       K(z, p)       \right|  
 $
is  bounded from above by a finite constant $\mathcal C_1>0$ for any $z  \in \Omega$, then the Bergman representative coordinate $T (z) $ relative to $p$  is biholomorphic from $\Omega$   to a disc possibly less a relatively closed  polar set. Moreover,  
 \begin{enumerate}
 \item [$1.$]    if
 $
  \left|       K(z, p)       \right|  
 $
is  bounded from below  by a finite constant $\mathcal C_2>0$ for any $z  \in \Omega$, then $T  $ has a biholomorphic inverse map that extends continuously up to the closure of  the disc;

\item [$2.$]   

if     $\Omega$ satisfies 
  local $C^1$-connectivity around any boundary point $q  
 \notin  \text{int} (\overline \Omega)$,  then $T  $ extends   continuously up to   $\overline\Omega $;

\item [$3.$]   

if both Part $1$ and Part  $2$ hold,  then $T  $ extends   to a homeomorphism of the closures.

 \end{enumerate}

\end{theorem}

          \medskip
    
 Theorem \ref{1-dim}    improves Theorem 1.3 in  \cite{DWo} for the case of planar domains. 
The conditions in   Theorem \ref{1-dim}  relate to an important  result \cite{Ker} of Kerzman, who used   Kohn's  theory of 
the $\bar \partial$-Neumann problem 
   to show that on a bounded strictly pseudoconvex domain $\Omega$ with
$C^\infty$-smooth boundary, for each fixed $p \in \Omega$, the Bergman kernel $K(\cdot, p)$ is $C^\infty$ up to the boundary.  In \cite[p.151-152]{Ker}, he gave an example of a simply-connected planar domain   whose Bergman kernel $K(\cdot, p)$ blows up to infinity at some boundary point; see also \cite{For} for an example of Forn\ae ss. 
  The  corollary below gives  sufficient conditions for the extension of the Riemann map  to the closures.

          \medskip

 \begin{cor} \label{simply} 
     Let $D$ be a bounded, simply-connected  domain in  $\mathbb C$.  Then, for any $p \in D$,  the Bergman representative coordinate $T (z)  $  relative to $p$   is   biholomorphic from $D$   to a disc $\mathbb D_r:= \{ w\in \mathbb C :     |w|^2    < 2g^{-1}(p) \}$, where $g$ is the Bergman metric of $D$. Moreover,   
     \begin{enumerate}
\item [$1.$]   
if     $\Omega$ satisfies 
  local $C^1$-connectivity around any $q  \in \partial D$ and
 there exists some point   $p \in D$ such that   
 $
  \left|       K(z, p)       \right|  
 $
is  bounded from above by a finite constant $\mathcal C_1>0$ for any $z  \in D$, 
then   $T  $ extends   continuously up to   $\overline D $;

\item [$2.$]   

  if
 $
  \left|       K(z, p)       \right|  
 $
is  bounded from below  by a finite constant $\mathcal C_2>0$ for any $z  \in D$, then $T  $ has a biholomorphic inverse map that extends continuously up to $\overline {\mathbb D_r}$;   in particular,  
 $\partial D$ is locally connected;

\item [$3.$]   

if both Part $1$ and Part  $2$ hold,   then $T  $ extends to a homeomorphism of the closures,  $\tilde {T}  : \overline D \to \overline{\mathbb D_r}$;  in particular,  
     $\partial D$ is a Jordan curve.

 \end{enumerate}

  \end{cor}

 \paragraph{Remark.} The condition   in Part 2 of  Corollary \ref{simply} is only sufficient but not necessary for the local  connectedness of the boundary, in view of  Webster's example (a) in \cite {Web}. Corollary \ref{simply}  also yields a criterion (see Proposition \ref {simply counter}) for   a  simply-connected planar domain $D$  to satisfy
$
\inf_{z \in D} |K(z, p)|  = 0,
 $
which may be compared with Webster's examples (a) and (b)   in \cite{Web}.

 \newpage

 Historically,  Bergman's  coordinates  provide a right tool to analyze the extension of biholomorphic maps to the boundary. Fefferman \cite{Fef}  proved that biholomorphic maps between two bounded domains  with smooth, strictly pseudoconvex boundaries in  $ \mathbb C^n$ admit  smooth extensions to the closures of the domains. Previously, it was known that such  maps   extend  to   homeomorphisms of the closures (see the papers of     Henkin \cite{Hen} and   Vormoor \cite{Vo}).  Webster \cite{Web}, and Nirenberg, Webster and Yang \cite{NWY} 
  gave   new proofs of Fefferman's mapping theorem, whose original proof used the asymptotic expansion of the Bergman kernel and an analysis of the boundary behavior of the geodesics of the Bergman metric.    
  Later, the Bergman-style coordinates were applied  by Bell and Ligocka   \cite{BLi, Lig} to prove that subelliptic estimates for the $\bar \partial$-Neumann problem imply boundary regularity of biholomorphic maps. 
 See also  \cite{Bell81, DF, For} for the extensions of biholomorphic maps involving more general domains. For more applications of  the Bergman representative coordinates, see the papers \cite{GK, Bell06, BCOV, Dinew, Berceanu, Yoo, Kra}  and the references therein.
 
        \medskip
            \medskip
            
   For higher dimensional domains,   we improve our results in \cite{DWo} by considering the following technical condition, which is similar to  the so-called  Condition $R$, cf. \cite{BLi, BB}. 

        \medskip
        
\begin{definition}   A domain $\Omega$ in $ \mathbb C^n, n\ge 1$, is said to satisfy Condition $(B)$ if  there exist  an open set   $U \subset \Omega$ 
and a finite constant $\mathcal C>0$ 
such that   for each $ j \in \{1, \dotsc, n\}$,  
 $$
  \left|  {   \frac{\partial }{\partial z_j   } K(z, p)   } \right|  \leq \mathcal C |K(z,  p)     |, \quad  \text{for any\, }    z  \in \Omega  \text{ and any\, }  p \in U.
   $$

  \end{definition}

          \medskip

  Condition $(B)$ is {\it not} a biholomorphic invariant, as can be seen from our Remarks in Section 3.
Our main  theorem
 gives a one-parameter family of biholomorphisms that extend homeomorphically to the closures.

        \medskip

     \begin{theorem} \label{with B}  Let $\Omega \subset \mathbb C^n, n \ge 1$, be a bounded pseudoconvex  domain whose Bergman metric $g$ has its holomorphic sectional curvature identically equal to a negative constant $-c^2$.  If  $\Omega$ satisfies 
   \begin{enumerate}
\item [$1)$]    Condition $(B)$ for some  open set   $U$, and

\item [$2)$]     local $C^1$-connectivity around any $q \in \partial \Omega$ such that
$
\limsup_{\Omega \ni z \to q} K(z, z) = \infty,
$
   \end{enumerate}
then for any  $p \in U$,  the Bergman representative coordinate $T (z)  $  relative to $p$ is biholomorphic from $\Omega$  to a ball 
\begin{equation} \label{ball}
\mathcal B:= \{ w\in \mathbb C^n :   \sum_{\alpha, \beta=1}^n  w_\alpha   g_{ \alpha \bar \beta } (p) \overline {w_\beta }   < {2}{c^{-2}} \}
\end{equation}  
 possibly less a relatively closed pluripolar set, where $n = 2 c^{-2}-1$,   and $T    $ extends to a homeomorphism of the closures.

  \end{theorem}

        \medskip
 
   \medskip

We say a set $E$ is pluripolar if there exists a plurisubharmonic function $\varphi$ in $\mathbb C^n$ such that $\varphi = -\infty$ on $E$. In view of a result \cite{S82} of Siciak, a pluripolar set is  
 negligible for $L^2$ holomorphic functions.  
It is known that a bounded $L^2$-domain of holomorphy is  pseudoconvex and its boundary  contains no pluripolar part, cf. \cite{PZ, I04}; see Example \ref{exam} for a pseudoconvex domain which is not an $L^2$-domain of holomorphy. 
 Theorem \ref {with B} directly yields the following corollary.

\begin{cor} \label{Cor}  Let $\Omega \subset \mathbb C^n, n \ge 1$, be a bounded  $L^2$-domain of holomorphy  whose Bergman metric $g$ has its holomorphic sectional curvature identically equal to a negative constant $-c^2$.  If  $\Omega$ satisfies 
 \begin{enumerate}
\item [$1)$]    Condition $(B)$ for some  open set   $U$, and
\item [$2)$]    local $C^1$-connectivity around any $q \in \partial \Omega$,
 \end{enumerate}
then for any  $p \in U$,  the Bergman representative coordinate $T (z)  $  relative to $p$ is biholomorphic from $\Omega$  to a ball 
$\mathcal B $ defined in  \eqref{ball},
where $n = 2 c^{-2}-1$,   and $T    $ extends to a homeomorphism of the closures,  $\tilde {T}  : \overline\Omega \to \overline{\mathcal B}$; in particular, the  boundary  $\partial \Omega$ is topologically the sphere $S^{2n-1}$.
\end{cor}

          \medskip
          
 For any bounded domain $\Omega$ that is biholomorphic to a ball in $ \mathbb C^n$, its Bergman metric is {\it complete} and has constant holomorphic sectional curvature.  
When such $\Omega$ satisfies both Condition $(B)$ and  local $C^1$-connectivity around any $q \in \partial \Omega$, our Corollary \ref{Cor} guarantees that   its  boundary  $\partial \Omega$ is topologically the sphere $S^{2n-1}$; the converse is not true, in view of Kerzman's example in \cite {Ker}.  Moreover, we prove that

         \medskip

\begin{theorem}  \label{under biholo}

Let $\Omega$  be a bounded domain that is  biholomorphic to   a ball in $ \mathbb C^n, n \ge 1$. 
 If   $\Omega$ satisfies    Condition $(B)$ for some  open set     $U$, and
 there  exists   a finite  constant  $    C>0$
 such that  
$$
  \left|       K(z, p)       \right|  \geq  C,   \quad     \, \forall z\in \Omega,  \, \forall p \in U,
$$
 then  for any $p \in U$, 
 the Bergman representative coordinate $T (z)  $ relative to $p$ has a biholomorphic    inverse map that extends continuously up to the closure of the ball. In particular,  $\partial \Omega$ is locally connected.
 
\end{theorem}

         \medskip

  Corollary \ref{Cor} and Theorem \ref {under biholo} can be regarded as multidimensional generalizations of 
  a theorem of   Carath\'eodory \cite{Cara}  on the continuous extension of biholomorphism up to the closures.
  Our last result  gives a Harnack-type estimate for the Bergman kernel on bounded domains in $\mathbb C^n$ with the   Bergman metric of constant holomorphic sectional curvature. 

           \medskip
           
 \begin{theorem}   \label{kernel similar} Let $  \Omega  $ be a bounded domain whose Bergman metric 
 has  constant holomorphic sectional curvature  in $ \mathbb C^n, n\geq 1$. 
 Then,     
  for any   $p \in \Omega$,   there exists  its small  neighborhood $U_p $   such that
\begin{equation}  \label{2 similar}
 \sup_{\zeta \in U_p } \left|   K(z, \zeta)\right|  \le 2  \left| K(z, p) \right|, \quad \text{for any\, }   z \in \Omega. 
\end{equation}

 \end{theorem}

    \medskip
    
 The organization of the paper is as follows. In Section 2, after estimating Bergman's coordinates, we prove our multidimensional results. In Section 3, we prove our one dimensional results.  In Section 4, we study the boundedness of the 
Bergman representative coordinate.

\section{Proofs of Multidimensional Results}

In this section, we characterize bounded   domains that are biholomorphic to a ball possibly less a relatively closed pluripolar set, and give sufficient conditions for the biholomorphisms to extend  continuously   up to the closures. 
We first give some estimates on Bergman's coordinates before proving our multidimensional results, which
generalize Carath\'eodory's theorem to higher dimensions and  link the aforementioned uniformization results of the Lu and Suita types.

 \subsection{Estimates of Bergman's Coordinates}  \label{2.1}

 Recall that a bounded domain $\Omega \subset \mathbb C^n$ is called a Lu Qi-Keng domain if for any  $p \in \Omega$, its Bergman kernel $K(\cdot , p)$  has no zero set, cf. \cite{Bo, Lu}.  The authors  in \cite{DWo}   showed that   
 if its (not necessarily complete) Bergman metric $g$ has constant holomorphic sectional curvature $-c^2$, then the domain $\Omega$ is Lu Qi-Keng and   the Bergman representative coordinate    
$T$ relative to $p$ maps  $\Omega$  to a ball
$
\mathcal B $  defined by \eqref{ball}. Previously, Lu's theorem in \cite{Lu} yields the same  conclusions  under the additional assumption that $\Omega$ is Bergman complete.  
One key step in \cite {DWo} was the use of Calabi’s concept of diastasis. Fix a point $z_0\in \Omega$ and let $A_{z_0}:=\{z\in \Omega \, | \,K(z, z_0) =0 \, \}$ be the zero set  of the Bergman kernel $K(\cdot , z_0)$. Since $A_{z_0}$ is an analytic variety, as domains $\Omega \setminus A_{z_0}$ and $\Omega$ have the same Bergman kernel $K$ and Bergman metric $g$.
Consider on $\Omega \setminus A_{z_0}$ the K\"{a}hler potential
\begin{equation} \label{dia}
\Phi_{z_0}(z):= \log \frac{K(z, z)  K(z_0, z_0)}{ |K(z, z_0)|^2} 
\end{equation}
for the Bergman metric $g=\partial \overline  \partial   \Phi_{z_0}$, and we call the function $\Phi_{z_0}(z)$ the Bergman-Calabi diastasis relative to $z_0$. The idea in  \cite{Lu, DWo} was to investigate locally the Taylor expansion of the K\"ahler potentials for the Bergman metric.
More precisely, at any $p \in \Omega$ with
   \begin{equation} \label{normal}
  g_{\alpha \bar \beta} (p) =\delta_{\alpha  \beta },
\end{equation}
there exists  a neighborhood $U_p$ such that the Bergman kernel on the diagonal can be   decomposed as
\begin{equation}  \label{K(z, z)}
K(z, z)=  \left (1 - \frac{c^2}{2} |T(z)|^2 \right )^{\frac{-2}{c^2}} e^{f(T(z))+\overline{f(T(z))}}, \quad z\in U_p,
\end{equation}
where $f$ is holomorphic on $U_p$.
Let  $\Omega^{\prime}:= \{z \in \Omega \setminus A_p : |T(z)|^2 < {2}{c^{-2}}  \}$ denote the set of points in $\Omega \setminus A_p$ that are mapped by $T$ into a ball 
 $ \mathbb B^n (0, {\sqrt{2} c^{-1} }):=\{(w_1, ... , w_n) :    |w|^2   < {2}{c^{-2}} \}$,
where $A_{p}$ is the zero set  of  $K(\cdot , p)$.  
  Then, by \eqref{K(z, z)} and the theory of power series, one may duplicate the variable with its conjugate so that  the full Bergman kernel can be complex analytically continued as
 $$
K(z, {z_0})= \left (1 - \frac{c^2}{2} \sum_{\alpha=1}^n  w_\alpha(z)     \overline {w_\alpha(z_0)}   \right )^{\frac{-2}{c^2}}   e^{f(T(z))+\overline{f(T({z_0}))}}, \quad z, z_0\in U_p.
$$
Moreover, for any $z, z_0\in U_p$,  it holds that
 $$
\Phi_{z_0}( z) 
  =  {\frac{-2}{c^2}}   \log  \left[    \left (1 - \frac{c^2}{2} |T(z)|^2 \right )  \left (1 - \frac{c^2}{2} |T(z_0)|^2 \right )  \left  | 1 - \frac{c^2}{2}   \sum_{\alpha=1}^n  w_\alpha(z)     \overline {w_\alpha(z_0)}   \right  |^{ -2}  \right  ].
   $$
Then, the symmetry of the Bergman-Calabi diastasis defined in \eqref{dia} further yields that
$$
\Phi_{p}(z_0)= \Phi_{z_0}(p) = {\frac{-2}{c^2}} \log  \left (1 - \frac{c^2}{2} |T(z_0)|^2 \right ).
$$
Since on  $\Omega^{\prime}$ both $\Phi_{p}(z)$  and  ${\frac{-2}{c^2}} \log  \left (1 - \frac{c^2}{2} |T(z )|^2 \right )$ are well-defined, these two real-analytic  functions coincide on  $U_p$, and thus are identical to each other on $\Omega^{\prime}$. That is, 
\begin{equation} \label{on Omega prime}
\Phi_{p}(z )=  {\frac{-2}{c^2}} \log  \left (1 - \frac{c^2}{2} |T(z )|^2 \right ), \quad z \in \Omega^{\prime}.
\end{equation} 
One can show by contradiction that  no point in  $\Omega \setminus A_p$ is mapped outside the ball $ \mathbb B^n (0, {\sqrt{2} c^{-1} }) $ by $T$, so $\Omega^{\prime}=  \Omega \setminus A_p $.
Therefore, \eqref{on Omega prime} in fact holds on $\Omega \setminus A_p $. By the Riemann removable singularity theorem, 
 $T$ extends across the analytic variety $A_p$ to the whole domain $\Omega$ with $|T(z)|^2 \leq {2}{ c^{-2}}$, and
  the maximum modulus principle yields  that $|T(z)|^2 < {2}{ c^{-2}}$  on $\Omega$. 
 The above explicit formula of $\Phi_{p}(z )$ also guarantees that  the zero set   $A_{z_0}=\emptyset$, as shown in \cite{DWo}, so $\Omega$ is a Lu Qi-Keng domain. The above
  methodology was further extended in \cite{DWW}.
 
 \medskip
 
    \medskip
    
   Based on these facts, in this subsection, we   obtain the following estimates.

    \medskip

 \begin{pro}  \label{upper bound} Let $\Omega \subset \mathbb C^n$ be a bounded domain whose Bergman metric $g$ has its holomorphic sectional curvature  identically equal to a negative constant $-c^2$. Then, for any $p \in \Omega$ satisfying \eqref{normal}, 
there exists a finite constant  $ C_p>0$  such that 
for each $\alpha, j = 1, \dotsc, n$,   it holds that
 \begin{equation} \label{good}
\left |K(z,  p)   \frac{\partial w_{\alpha} (z)}{\partial z_j}  \right | \leq \left |  \frac{\partial    }{ \partial {z_j}   } \left(   \left. \frac{\partial }{  \partial  \overline {t_\alpha} } \right|_{t=p}     K(z, t)       \right)  \right| + C_p  \left|  {   \frac{\partial K(z, p)}{\partial z_j   }    } \right|, \quad \forall z\in \Omega.
\end{equation}
If  $\Omega$ additionally satisfies   Condition $(B)$ for some  open set   $U$, then for any $p \in U$ that satisfies \eqref{normal}, each component of  the Bergman representative coordinate $T$ relative to $p$ has its partial derivatives  bounded from above  on $\Omega $ by some finite positive constant.

\end{pro}

\begin{proof}

By   the previous discussion,   the Bergman representative coordinate $T$ relative to $p$ which satisfies \eqref{normal} maps $\Omega$ to a ball 
 $ \mathbb B^n (0, {\sqrt{2} c^{-1} }) $. It then follows that
$$
\sum_{j =1}^n  \left | K^{-1}(z, p)  \left. \frac{\partial}{\partial \overline {t_j} } \right|_{t=p} K(z, t)  -    \left. \frac{\partial}{\partial \overline {t_j} } \right|_{t=p}  \log K(t, t)\right |^2   < {2}{c^{-2}}, \quad \forall z\in \Omega,
$$
which further implies that
$$
\sum_{j =1}^n  \left|  K^{-1}(z, p)  \left. \frac{\partial}{\partial \overline {t_j} } \right|_{t=p} K(z, t)  \right|  \leq  \sum_{j =1}^n  \left|    \left. \frac{\partial}{\partial \overline {t_j} } \right|_{t=p}  \log K(t, t) \right| +  n {\sqrt 2}{c^{-1}}  =:C_p, \quad \forall z\in \Omega.
$$
Here $C_p$ is a finite positive constant depending on $p$ since the Bergman kernel is locally uniformly bounded.
Thus, for each $j = 1, \dotsc, n$,
\begin{equation} \label{less}
\left|  \left. \frac{\partial}{\partial \overline {t_j} } \right|_{t=p} K(z, t)  \right|  \leq  C_p |K(z, p)|, \quad \forall z\in \Omega.
\end{equation}
By the definition  \eqref{rep}, we know that for each $\alpha, j = 1, \dotsc, n$,  
\begin{align*}
\frac{\partial w_{\alpha} (z)}{\partial z_j}  &=  \frac{\partial }{\partial z_j} \left\{K^{-1}(z, p) \left. \frac{\partial}{\partial \overline {t_\alpha} } \right|_{t=p} K(z, t)\right\} \\
&= \frac{ K(z, p)  \frac{\partial    }{ \partial {z_j}   } \left(   \left. \frac{\partial }{  \partial  \overline {t_\alpha} } \right|_{t=p}     K(z, t)       \right)  -   \frac{\partial K(z, p)}{\partial z_j   }  \left. \frac{\partial }{  \partial \overline {t_\alpha} } \right|_{t=p}  K(z, t)    }{K^2(z, p)}, \quad \forall z\in \Omega.
\end{align*}
This combined with  \eqref{less} will yield  that for any $z\in \Omega$, 
\begin{align*}
\left |K(z,  p)   \frac{\partial w_{\alpha} (z)}{\partial z_j}  \right | & \leq  \left |  \frac{\partial    }{ \partial {z_j}   } \left(   \left. \frac{\partial }{  \partial  \overline {t_\alpha} } \right|_{t=p}     K(z, t)       \right) \right |  +  \left |  \frac{ \frac{\partial K(z, p)}{\partial z_j   }  \left. \frac{\partial }{  \partial \overline {t_\alpha} } \right|_{t=p}  K(z, t)  }  {K(z, p) } \right | \\
& \leq   \left |    \frac{\partial    }{ \partial {z_j}   } \left(   \left. \frac{\partial }{  \partial  \overline {t_\alpha} } \right|_{t=p}     K(z, t)       \right) \right |  +  C_p \left |   \frac{\partial K(z, p)}{\partial z_j   }      \right |.
\end{align*}

If  $\Omega$ additionally satisfies   Condition $(B)$ for some  open set   $U$, then for any $p=(p_1, \dotsc, p_n) \in U \subset \Omega$ that satisfies \eqref{normal},
take a small polydisc $\mathbb D^n(p; r_p)  \subset U$  for some  $ r_p >0$.
By Cauchy's integral formula for derivatives,  for each $\alpha, j = 1, \dotsc, n$, it holds that
  \begin{align*}
 \left|  \left. \frac{\partial }{  \partial  {t_\alpha} } \right|_{t=p}   \left(  \frac{  \frac{\partial    }{ \partial \overline {z_j}   } K(t, z) } {K(t, z)}    \right)  \right| 
& =  \left|  \frac{1}{2\pi i} \int_0^{2\pi} \left( \frac{\frac{\partial   }{ \partial \overline {z_j}   } K( (p_1, \dotsc ,  p_\alpha +  r_p e^{i \theta}, \dotsc,   p_n), z)}{K((p_1, \dotsc ,  p_\alpha +  r_p e^{i \theta}, \dotsc,   p_n), z)} \right)  \frac{ r_p e^{i \theta} i d \theta}{(r_p e^{i \theta})^2   }      \right|\\
& \leq  \frac{1}{r_p}  \sup_{t \in U } \left|   \frac{ \frac{\partial    }{ \partial \overline {z_j}   } K( t , z) } {K( t , z) }  \right|\\
& \leq  \frac{\mathcal C}{r_p}, \quad \forall  z\in \Omega.
\end{align*} 
  Thus,  by assumption,   
\begin{align*}
\left| \frac{ \left. \frac{\partial }{  \partial  {t_\alpha} } \right|_{t=p}   \left(    \frac{\partial    }{ \partial \overline {z_j}   } K(t, z)       \right) } {K(p, z)}  \right| & = \left|  \left. \frac{\partial }{  \partial  {t_\alpha} } \right|_{t=p}   \left(  \frac{  \frac{\partial    }{ \partial \overline {z_j}   } K(t, z) } {K(t, z)}    \right) + \frac{ { \frac{\partial    }{ \partial  \overline {z_j}   } K(p, z) } }{K (p, z)}     \frac{ \left. \frac{\partial }{  \partial  {t_\alpha} } \right|_{t=p}  K(t, z)}{K(p, z)}\right|\\
& \le  \frac{\mathcal C}{r_p} + \mathcal C \cdot C_p,    \quad \forall z\in \Omega.
\end{align*}  
Therefore, by  Estimate \eqref{good},   it holds  that 
 \begin{align*}  
\left |  \frac{\partial w_{\alpha} (z)}{\partial z_j}  \right | & \leq \left| \frac{ \left. \frac{\partial }{  \partial  {t_\alpha} } \right|_{t=p}   \left(    \frac{\partial    }{ \partial \overline {z_j}   } K(t, z)       \right) } {K(p, z)}  \right| + C_p  \left| \frac{    \frac{\partial K(z, p)}{\partial z_j   }    }  {K(z,  p) }\right|\\
&   \leq  \frac{\mathcal C}{r_p} + 2\mathcal C \cdot C_p,  \quad \forall z\in \Omega.
\end{align*}  

 \end{proof}

      \medskip
     
Compared with  Definition \ref{def-local connect}, there is a stronger local convexity.

   \begin{definition} \label{def-local convex}
Let $\Omega$ be a bounded domain in  $ \mathbb C^n, n \ge 1$.
We say $\Omega$ is  locally $C^1$-diffeomorphic to a convex domain around   some $q \in \partial \Omega$, 
if
there exists an open set $U \ni q$ and a $C^1$ diffeomorphism $\varphi: U \to \mathbb R^{2n}$ such that 
$\varphi(U \cap \Omega)$ is a convex set in $\mathbb R^{2n}$.
  \end{definition}

  For  a domain $\Omega \subset  \mathbb C^n$ with piecewise $C^1$-smooth boundary, take a boundary point  $q \in \partial \Omega$. Then there exists a neighborhood  $U \ni q$ in $\mathbb C^n$ and a $C^1$-diffeomorphism $\varphi : U \to \mathbb B^n$ such that $\varphi (U \cap  \Omega) =  \mathbb B^n  \cap \Omega^{\prime}$, where $\Omega^{\prime}$ is a convex polyhedron. Thus,  according to Definition \ref {def-local convex}, such $\Omega$ is  locally $C^1$-diffeomorphic to a convex domain around  $q $. Moreover, we get
   
      \medskip
      
  \begin{pro}  \label{stronger} For  a bounded domain $\Omega$   in  $ \mathbb C^n, n \ge 1$, if it  is  locally $C^1$-diffeomorphic to a convex domain around   some $q \in \partial \Omega$, then 
 it satisfies  local $C^1$-connectivity around   $q.$
     \end{pro}

 \begin{proof}  Denote the $C^1$-diffeomorphic image of $U$ by $V:=\varphi (U)$.  Shrink $V$ if necessary so that there exist constants $m_1, m_2>0$ with 
 \begin{equation} \label{d varphi}
m_1 < \frac{|d \varphi (W)|}{|W|} < m_2, 
\end{equation}  
 for any non-zero vector $W$ in the tangent space of any $z \in (U \cap \Omega)$.
For any $\epsilon >0$, choose $U_\epsilon \ni q $, a  subset  of  $U$, such that the  diameter of the convex set 
$V_\epsilon:= \varphi (U_\epsilon  \cap \Omega)$  is less than $ m^{-1}_1 \epsilon$.
 For any  $x, y \in (U_\epsilon  \cap \Omega)$, join $\varphi(x)$ and $\varphi(y)$ in the convex set $V_\epsilon$ by a straight line called $L$.  Then, $\gamma_{x, y}:= \varphi^{-1}(L)$ would be a  $C^1$ path  in $(U_\epsilon  \cap \Omega)$ joining $x$ and $y$.
 Moreover, by \eqref{d varphi}, the length of  $\gamma_{x, y}$, denoted by $\|\gamma_{x, y}\|$  would satisfy
 \begin{equation} \label{equi}
 m_1 \|\gamma_{x, y}\|  \leq  |\varphi(x) - \varphi(y)|  \leq  m_2 \|\gamma_{x, y}\|.
 \end{equation} 
 Since $|\varphi(x) - \varphi(y)|$ is less than $ m^{-1}_1 \epsilon$, it holds that $\|\gamma_{x, y}\|< \epsilon$.

 \end{proof}

     \medskip

          \begin{remark}   \label{Example C^1 conn}

        Recall that 
a domain $\Omega$ is strictly pseudoconvex near $q \in \partial \Omega$ if and only if $\Omega$ is strictly Euclidean convex with respect to suitable local holomorphic coordinates centered at $q \in \partial \Omega$.
Thus, if $\Omega \subset \mathbb C^n$ is a convex domain, a strictly pseudoconvex domain or a domain with piecewise $C^1$-smooth boundary, then $\Omega$ is  locally $C^1$-diffeomorphic to a convex domain around any boundary point $q \in \partial \Omega$. Consequently, by Proposition \ref{stronger},  a bounded domain   satisfies  local $C^1$-connectivity around  any boundary point, if it is a convex domain, a strictly pseudoconvex domain or a domain with piecewise $C^1$-smooth boundary.
See also Example \ref {counter} for   a domain not satisfying  local $C^1$-connectivity.

        \end{remark}

\subsection{Proof of the Main Theorem}

The   theorem  below will play a crucial role in the proof of our Theorem \ref{with B}.

     \medskip

\begin{theorem} {\normalfont\cite{DWo}}   \label{biholo} Let $\Omega \subset \mathbb C^n$ be a bounded domain whose Bergman metric has its holomorphic sectional curvature  identically equal to a negative constant $-c^2$. Then, for some  $z_0\in \Omega$ , $\Phi_{z_0}(z)$ blows up to infinity at $\partial \Omega$   if and only if   $\Omega$ is biholomorphic to a ball  and $n = 2/c^2-1$.
    
\end{theorem}

    \medskip

  Under the constant holomorphic sectional curvature assumption, two proofs of the above equivalence using and not using Lu's theorem were given in \cite{DWo}. For a general bounded domain $\Omega$, as demonstrated  in  \cite[Proposition 3.1]{DWo}, the fact that for some  $z_0\in \Omega$, $\Phi_{z_0} $ blows up to infinity at $\partial \Omega$  does not imply the completeness of the Bergman metric of $\Omega$.

 \begin{proof} [{\bf Proof of Theorem \ref{with B}}]
 
 {\bf Step 1, conclusion of a biholomorphism.}   For any boundary point $q\in \partial \Omega$ such that 
 \begin{equation} \label{finite}
 \limsup \limits_ {\Omega \ni z\to q} K(z, z) < \infty,
 \end{equation}  
a result of Pflug and Zwonek \cite{PZ} says that $q \in \text{int} (\overline \Omega)$ and there exists a neighborhood $U$ of $q$ such that $P:=U \setminus \Omega$ is a pluripolar set. 
In particular, both the  Bergman kernel $K(z, p)$ and the Bergman representative coordinate $T (z) =(w_1, ... , w_n)$  extend across $q$.
Also, we {\bf claim} that  
the pluriharmonic function 
$$h_p(z):= 2 \log   |K(z, p)|$$
satisfies  
 \begin{equation} \label{not infit}
\liminf_{z \to q} h_p(z) > -\infty.
 \end{equation} 
If not, then there exists a sequence of points $ (z_j)_{j\in \mathbb N} \subset \Omega$ approaching $q$ such that
$$
\lim_{j \to \infty} |K(z_j, p)| = 0.
$$
By the continuity of $K(z, p)$, this forces $K(z, p)$ to attain zero at the extended point $q$, which is impossible in view of \cite{DWo}.
Consequently,  by the above claim,
$h_p(z) $  extends across $q$.

\medskip

 Let $w \in \partial \Omega$ be an arbitrary boundary point  such that 
$$
\limsup_{\Omega \ni z \to w} K(z, z) = \infty.
$$
  Condition $(B)$ implies that  for each $  j \in \{1, \dotsc, n\}$,
    $$
 \left|      \frac{\partial    h_p}{\partial z_j   }     (z)   \right| =  \left|      \frac{\partial \log   (K(z, p) \overline {K(z, p)})}{\partial z_j   }       \right| =    \left|  \frac{   \frac{\partial }{\partial z_j   }  K(z, p)  }  { K(z, p)} \right|     \leq \mathcal C, \quad \forall z  \in \Omega.
$$
Using the  local $C^1$-connectivity around $w$,    one  shows that for a piece-wise $C^1$ path $\gamma_{x,y}$ lying  in a neighborhood of $w$ and joining $x$ and $y$, its length
  $\|\gamma_{x,y}\|$  would satisfy 
  $$
  |h_p(x) -h_p(y)| \leq   \mathcal C     \|\gamma_{x,y}\| \to 0,
  $$ 
   as $x$ and $y$ both tend to $w$. Thus, 
   $h_p (z) $ is continuous up to  $w$, and by \eqref{not infit}
 we conclude that $h_p (z) $ is continuous up to  the whole $\overline\Omega $. Then, there exists a finite constant $\tilde{\mathcal C}_p>0$ such that
$|h_p | < \tilde{\mathcal C}_p$ on $\overline  \Omega$. Therefore, 
 \begin{equation} \label{upper bound K}
|K(z, p)| = e^{\frac{h_p (z)}{2}} \le e^{\frac{|h_p (z)|}{2}} \le e^{\frac{\tilde{\mathcal C}_p}{2}} =:  {\mathcal C}_p \in \mathbb R^+, \quad \forall z\in \Omega,
 \end{equation}
which yields that   
 $$
   \limsup \limits_ {\Omega \ni z\to w} \Phi_{p}( z) \geq   \limsup \limits_ {\Omega \ni z\to w}\log \frac{K(z, z) K(p, p)} { \mathcal C_p ^2} = \infty.
 $$

\medskip
 
Next, we   assume that there exists some point $p=(p_1, \dotsc, p_n) \in U \subset \Omega$   with \eqref{normal}.  
By Proposition \ref{upper bound} and the  local $C^1$-connectivity around $w$, one can show that for $x, y \in \Omega$ close to $w$,
 $$
 |T(x) -T(y)| \leq ( \frac{\mathcal C}{r_p} + 2\mathcal C \cdot C_p)  \|\gamma_{x,y}\|.
 $$
Similarly,   $T (z) $ is continuous up to $w$ and thus up to the whole  $\overline\Omega $.
 Therefore, for any two sequences of points $ (z_j)_{j\in \mathbb N}, (w_j)_{j\in \mathbb N} \subset \Omega$, both approaching $w$, it holds that
 $$
\lim_{j \to \infty}T(z_j)=\lim_{j \to \infty}T(w_j).
 $$
  By the   explicit formula  of  the Bergman-Calabi diastasis  
     \begin{equation}  \label{formula}
\Phi_{p}( z) =  {\frac{-2}{c^2}} \log      \left (1 - \frac{c^2}{2} |T(z)|^2 \right ), \quad z \in  \Omega,
  \end{equation} 
we know that
 $$
\lim \limits_ {\Omega \ni z\to w} \Phi_{p}( z)= \infty.
$$

\medskip

Taking all boundary points $q_j$ that satisfy \eqref{finite}, we get the corresponding neighborhoods $U_j$ and pluripolar sets $P_j$. Then the (bounded) domain
$$
\tilde\Omega:= \bigcup_{j}U_j \cup \Omega
 $$
has the same Lebesgue measure as $\Omega$ due to the pluripolarity. In view of \cite{PZ}, $\partial \tilde \Omega$ coincides with the non-pluripolar part of $\partial \Omega$.   Moreover, 
as domains $\tilde\Omega$ and $\Omega$ have the same Bergman metric (with constant holomorphic sectional curvature). To see this, notice that $P_j=U_j \cap \partial \Omega$ is relatively closed in $U_j$. 
So, restricting any function $f \in L^2 \cap \mathcal O (\Omega)$ to $U_j\setminus P_j$, we get  a function $F \in L^2 \cap \mathcal O (U_j)$ such that $F=f$ on $U_j\setminus P_j$. Hence one has an
$L^2$ holomorphic extension to $\Omega \cup U_j$, and consequently to $\tilde\Omega$.
By the above discussions, 
we know that the Bergman-Calabi diastasis $\Phi_{p}$  blows up to infinity at $\partial \tilde \Omega$, since  $w$ is arbitrary. 
Then, Theorem \ref{biholo} guarantees that $\tilde\Omega$ is biholomorphic to a  ball  
 and $n = 2/c^2-1$. 
Define the set 
$$
E:=\bigcup\limits_{j}P_j = \bigcup\limits_{j} U_j \cap \partial \Omega = \tilde\Omega \cap \partial \Omega,
$$ 
which is  relatively closed in $\tilde\Omega$.
Since the biholomorphic (pre)images and countable union of pluripolar sets are still pluripolar, we conclude that $T (z)  $ is biholomorphic from $\Omega$  to a ball 
 $ \mathbb B^n (0, {\sqrt{2} c^{-1} })$ 
 possibly less a relatively closed pluripolar set $E$.

 \medskip

Finally, if $p $ is some general point in the open set $U$ for which  $\Omega$ satisfies Condition $(B)$, let  
    $g_{\alpha \bar \beta} (p)$ be the positive-definite  Hermitian matrix associated with the Bergman metric at $p$.
One performs a   linear transformation $F$ from $\Omega$ to  $\Omega^1$ such that the Bergman metric $g^1$ on $\Omega_1 $  satisfies 
$
  g^1_{\alpha \bar \beta} (F (p)) =\delta_{\alpha  \beta }.
  $
 Since $F$ is a biholomorphism, the Bergman metric on $\Omega^1 $ also has  its holomorphic sectional curvature identically equal to a negative constant $-c^2$.  
   Moreover, the Bergman kernels on $\Omega$ and $\Omega^1$  differ by a multiple constant,
  which is the determinant square of  $g_{\alpha \bar \beta} (p)$, due to the transformation rule. Therefore, $\Omega^1$ satisfies Condition $(B)$ for the open set $F (U) $. By the previous argument,   the Bergman representative coordinate $T^1 (z)  $   relative to $F (p)$ is 
 biholomorphic from $\Omega^1$  to a ball  $\mathbb B^n_r$
 possibly less a relatively closed pluripolar set, and  extends  continuously up to   $\overline{\Omega^1}$. Finally, the composition map $F^{-1} \circ T^1 \circ F$ is the Bergman representative coordinate $T (z)  $   relative to $ p$ and it is 
 biholomorphic from $\Omega$  to a ball  $\mathcal B$  defined in \eqref{ball}
 possibly less a relatively closed pluripolar set $E$.   The  continuous extension  up to   $\overline{\Omega}$
   follows due to the   linearity of   $F$.

  \medskip

 {\bf Step 2, extension as  a homeomorphism.} Since the pluripolar set $E$ is negligible for $L^2$ holomorphic functions, the  transformation formula of the Bergman kernel yields that 
\begin{equation} \label{transf}
K(z, p)  = D_{T} (z) \cdot \overline{D_{T}(p) }\cdot K_{\mathcal B \setminus E}\left(T (z),  0\right) = D_{T} (z) \cdot  K_{\mathcal B}\left(0,   0\right)= \frac{D_{T} (z)} {v( \mathcal B)},   \quad \forall  z \in \Omega,
\end{equation}
where $D_{T} $ is the determinant of the complex Jacobian of $T$ and $v(\cdot)$ denotes the Euclidean volume. The second equality in \eqref{transf} holds  due to the explicit formula of the Bergman kernel on the ball $\mathcal B$.
   Let $S$ be the inverse map of $T $. 
By the inverse function theorem, 
the complex Jacobian   of $S$ is the inverse  of the complex Jacobian   of $T$, and
\begin{equation} \label{inverse}
JS = (JT)^{-1} = \frac{\text{Adj}(JT)}{D_{T} },
\end{equation} 
where $Adj (JT)$ is the adjugate matrix of $JT$, the complex Jacobian of $T$.  By Proposition \ref{upper bound}, each entry of $Adj (JT)$ is bounded from above  on $\Omega $ by some finite positive constant.
 By \eqref{upper bound K}, we know that
$$
|K(z, p)|^{-1} = e^{-\frac{h_p (z)}{2}} \le e^{\frac{|h_p (z)|}{2}} \le    {\mathcal C}_p, \quad \forall z\in \Omega,
$$ 
which means
 $$
 |K(z, p)|  \ge {\mathcal C}_p^{-1} >0, \quad \forall z\in \Omega.
 $$
Therefore, by  \eqref{transf},
$$
| {D_{T} (z)} |= {|K(z, p)| v( \mathcal B)} \geq  { {\mathcal C}_p^{-1} }  {v( \mathcal B)} , \quad \forall z \in \Omega,
$$
so each component  of  $S$ is a bounded holomorphic function whose partial derivatives  are bounded  from above by some  finite constant  on $\mathcal B \setminus E$.
 Since $S$ extends across $E$, and the boundary of $\mathcal B$ is smooth, 
by an elementary calculus argument  involving the uniform continuity, we  conclude that $S $ is Lipschitz and thus continuous up to   $\overline {\mathcal B }$. This further implies that $T    $ extends to a homeomorphism of the closures.

   \end{proof}

\noindent{}{\bf Remarks.}   
       
               \medskip
               
\begin{enumerate}[label=(\roman*)]

\item 
In Theorem \ref{with B}, by the continuity of $T^{-1}$, the inverse of the Bergman representative coordinate  $T$,   we can see that  $T (\partial \Omega) = \partial ({\mathcal B \setminus E})$.
If $T$ maps two points  $r, s \in \partial \Omega$   to a single point  in $ \partial ({\mathcal B \setminus E})$,  let $\{r_j\}_{j\in \mathbb N}, \{s_j\}_{j\in \mathbb N} $ denote two sequences of interior points in  $\Omega$ such that 
$$
\lim_{j\to \infty} r_j =r, \quad \lim_{j\to \infty} s_j =s.
$$
By the continuity of the Bergman representative coordinate  $T$ and 
its  inverse $T^{-1}$,
$$
\lim_{j\to \infty}   T(r_j) =  T(r)= T(s) = \lim_{j\to \infty}   T(s_j),
$$
and thus 
$r= s$, which implies that  $T:  \overline{\Omega} \to \overline  {\mathcal B \setminus E}$ is  one-to-one.

  \medskip

\item 

 In our proof of Theorem \ref{with B},   it suffices to require the path $\gamma_{x, y}$ in Definition \ref{def-local connect} to be rectifiable.

\item 

   The pseudoconvexity in Theorem \ref{with B}  is a necessary assumption. To see this, remove from the unit ball $\mathbb B^n, n \geq 2$, a compact subset $G$ of Lebesgue $\mathbb R^{2n}$-measure zero such that $\mathbb B^n \setminus G$ is connected. Then, the Bergman metric on $\mathbb B^n \setminus G$ extends to $\mathbb B^n$ by Hartogs' extension theorem. So, the assertion of Theorem \ref{with B} fails if 
$G$ is not pluripolar.

\end{enumerate}

         \medskip

 The proof of Theorem \ref{with B} also yields the following result.

            \medskip
            
\begin{theorem}   Let $\Omega \subset \mathbb C^n, n \ge 1$, be a bounded pseudoconvex  domain whose Bergman metric $g$ has its holomorphic sectional curvature identically equal to a negative constant $-c^2$.  If   for any $q \in \partial \Omega$ with 
$$
\limsup_{\Omega \ni z \to q} K(z, z) = \infty,
$$
it holds that
$$
\lim \limits_ {\Omega \ni z\to q} \Phi_{p}( z)= \infty,
$$
  for some  $p \in \Omega$, 
  then $n = 2 c^{-2}-1$ and the Bergman representative coordinate $T (z)  $  relative to $p$ is biholomorphic from $\Omega$  to a ball  possibly less a relatively closed pluripolar set.

\end{theorem}

\subsection{$L^2$-Domain of Holomorphy}

    \medskip

Corollary \ref{Cor} follows directly from  Theorem \ref{with B} as  the boundary of  a bounded $L^2$-domain of holomorphy,   which is the domain of existence of some $L^2$ holomorphic function, contains no pluripolar part, cf. \cite{PZ, I04}.  Moreover, it is well-known that a  bounded domain which is complete with respect to the Bergman metric  is an $L^2$-domain of holomorphy.

     \medskip

 For a bounded domain $\Omega$  in $\mathbb C^n$ with a {\it complete} Bergman metric of constant holomorphic sectional curvature, although Lu's theorem says that the Bergman representative coordinate $T$ is a   biholomorphism that  maps  $\Omega$  to $\mathcal B $, 
 our Corollary \ref{Cor} guarantees that  the boundary  $\partial \Omega$ is topologically the sphere $S^{2n-1}$
 whenever such $\Omega$ satisfies both Condition $(B)$ and  local $C^1$-connectivity around any $q \in \partial \Omega$.

    \medskip

 \begin{proof}   [{\bf Proof of Theorem \ref{under biholo}}]
 After a possibly linear change of coordinates, 
 it suffices to  assume that   at      some   $p \in U$,  \eqref{normal}   holds true. 
Then, take a small polydisc $\mathbb D^n(p; r_p)  \subset U$  for some  $ r_p >0$.
By Proposition \ref{upper bound}, we know that 
 for each $\alpha, j = 1, \dotsc, n$, 
  $$
\left |  \frac{\partial w_{\alpha} (z)}{\partial z_j}  \right |  \leq  \frac{\mathcal C}{r_p} + 2\mathcal C \cdot C_p     < \infty, \quad \forall   z\in \Omega.
 $$
 Let 
 $D_{T} $ be the determinant of the complex Jacobian $JT$, and let $Adj (JT)$ be the adjugate matrix of $JT$.
 Then, each entry of $Adj (JT)$ is bounded from above  on $\Omega $ by some finite positive constant.
  Let $S$ be the inverse map of $T $. 
By the inverse function theorem, 
the complex Jacobian   of $S$ is the inverse  of the complex Jacobian   of $T$  and satisfies \eqref{inverse}.
 Also, by  \eqref{transf},
$$
| {D_{T} (z)} |= {|K(z, p)| v( \mathcal B)} \geq  { C }  {v( \mathcal B)} , \quad \forall z \in \Omega,
$$
 where $v(\cdot)$ denotes the Euclidean volume. 
 Therefore, each component  of  $S$ is a bounded holomorphic function whose partial derivatives  are bounded  from above by some  finite constant  on the ball $\mathcal B $. We  then conclude that $S $ is Lipschitz and thus continuous up to   $\overline {\mathcal B }$. By the continuity of $S$, it follows that $S( \partial \mathcal B) = \partial \Omega$, so $\partial \Omega$ is locally connected.

 \end{proof}

     \medskip

The following is an  example of a pseudoconvex domain which is not an $L^2$-domain of holomorphy. 
It was provided to us by Peter Pflug, and we present it here  with his kind permission.

     \medskip

            \begin{example} \label{exam}
            $\mathcal D:=\{ (z, w) \in \mathbb C^2 : |z|^2+|w|^2<1,  z \neq 0\}$ is an example of a {\it pseudoconvex} domain that is not an $L^2$-domain of holomorphy. To see this, let $\mathbb B^2$ be the unit ball in $\mathbb C^2$ and let $E := \{ (0, w) \in \mathbb C^2 : |w|<1 \}$ be a relatively closed  pluripolar subset of  $ \mathbb B^2$. Then, $\mathcal D = \mathbb B^2 \setminus  E$  and $\mathcal D$ is not an $L^2$-domain of holomorphy since all $L^2$ holomorphic functions on $\mathcal D$ extend across $E$.
         Indeed, $\mathcal D$ is  pseudoconvex. For instance, the function 
 $f(z, w)=1/z$ is  holomorphic on $\mathcal D$, but $f$ cannot extend across $E$ on which it blows up.

   \end{example}

\section{Proofs of One Dimensional  Results}

In this section, we prove our one dimensional  results, 
based on the authors' previous work  \cite {DWo} as well as properties of Green's functions.
 Let $SH^{-}(X)$ denote the set of negative subharmonic functions on an open Riemann surface $X$. Given   $z_0 \in X$,  let $w$ be  a fixed local coordinate  in a neighbourhood of $z_0$ such that $w(z_0) =0$. The (negative) Green's function   is
\begin{equation} \label{Other Green's Function}
G(z, z_0) = \sup\{u(z): u \in SH^{-}(X), \limsup_{z \to z_0} u(z) - \log|w(z)| < \infty\}.
\end{equation}
 An open Riemann surface admits a Green's function if and only if there is a non-constant, negative subharmonic function defined on it.  The Green's function  is strictly negative on the surface and harmonic except on the diagonal.   The Green's function on $X$ induces the logarithmic capacity $c_{\beta}$, which is defined as
 $$
c_{\beta}(z_0) = \lim_{z \to z_0} \exp(G(z, z_0) - \log|w(z)|),
 $$
 where  $w$ is  a fixed local coordinate  in a neighborhood of $z_0 \in X$ with $w(z_0) =0$. Denote  by $K$ the Bergman kernel   on the diagonal for holomorphic 1-forms on $X$. 
Suita's conjecture (now a theorem of B\l{}ocki \cite{Bl13}, and Guan and Zhou \cite{GZ}) asserts that for any surface $X$ as above,  it holds that
\begin{equation} \label{leq}
\pi K  \geq c_{\beta}^2, 
\end{equation} 
and equality in \eqref{leq} holds true at some $z\in X$ if and only if $X$ is  {\it biholomorphic to a disc possibly less a relatively closed polar set.}
We say that a Borel set $P$ is polar if there is a subharmonic function $u \not\equiv -\infty$ defined in a neighbourhood of $P$ so that $P \subset \{z: u(z) = -\infty\}$.

     \medskip
     \medskip

  Due to an identity 
$
  (\log c_{\beta} )_{ z  \bar z}=\pi K
$
 proved by Suita in the same paper \cite{Su}, his conjecture has  geometric interpretation: 
the inequality part in \eqref{leq} is equivalent to  that the Gaussian curvature of the invariant metric $c_{\beta}^2(z)dz\otimes d\overline{z}$ always has an upper bound $-4$;
the equality/rigidity part, which says that the curvature attains $-4$ at some $z  \in X$, guarantees that the surface is necessarily as asserted with the curvature identically $-4$.  Similar curvature property is satisfied for the Carath\'eodory metric (see \cite{Su73, W}); additional results on 
metrics of constant Gaussian curvature were obtained  in \cite{D}.  
Notice that neither boundary regularity  nor completeness of  the metric  is  assumed in the Suita type problems. 
 Moreover, B\l{}ocki and Zwonek  \cite{BZ15} obtained a multidimensional version of the Suita conjecture (see also \cite{BBMV, GZa} for related comparison  results).

     \medskip

 \medskip

For a bounded domain $\Omega \subset \mathbb C $, there is an equivalent definition of the Green's functions on it, cf. \cite[Definition 4.4.1]{R95}.

     \medskip

 \begin{definition}  \label{Green def} A Green's function for $\Omega$ is a map $g_\Omega: \Omega \times \Omega \to [-\infty, \infty)$ such that for each $p\in \Omega$,
 
  \begin{enumerate}

\item [$(a)$] $g_\Omega(\cdot, p)$ is harmonic on $\Omega  \setminus \{p\}$, and bounded outside each neighbourhood of $p$;

\item [$(b)$] $g_\Omega(p, p) = -\infty$, and as $z\to p$, 
$$
g_\Omega(z, p) = \log |z-p| +\Hi (1);
$$

\item [$(c)$] $g_\Omega(z, p) \to 0$, as $z \to \zeta$ for n.e. $\zeta \in \partial \Omega$.
 
  \end{enumerate}

  \end{definition}

           \medskip
           
  Here, a property is said to hold nearly everywhere (n.e.) on a subset $S $ of $\mathbb C$ if it holds on $S \setminus E$, for some Borel polar set $E$, cf. \cite[Definition 3.2.2]{R95}.  
Recall that  for a domain $\Omega \subset \mathbb{C}$, a point $\zeta_0 \in \partial \Omega$ is said to be irregular if there is a $\zeta \in \Omega$ such that $\lim_{z \to \zeta_0} G(z, \zeta) \neq 0$ and regular otherwise.  By Kellogg's Theorem, cf. \cite[Theorem 4.2.5]{R95}, the set of irregular boundary points of a domain is a polar set.  
The following useful lemma could be found, for instance,  in  \cite{DTZ}.

     \medskip

 \begin{lem}[Rigidity lemma of the Green's function] \label{Green lem}

On an open Riemann surface $X$, the Green's function with a pole $z_0 \in X$ is
 $$
G(z, z_0)= \log|f(z)|
 $$
for some holomorphic function $f$ on $X$ if and only if $f$ is a biholomorphism from $X$ to the unit disk possibly less a relatively closed polar subset such that $f(z_0) =0$.

\end{lem}

     \medskip

 \begin{proof} [{\bf Proof of Theorem \ref{1-dim}}] Since the Bergman metric  on $\Omega$ has Gaussian curvature identically equal to $- 2$, by  \cite {DWo} we know that the Bergman representative coordinate $T (z) $ relative to $p$  is  a bounded holomorphic map on $\Omega$ with  
 $|T |^2<  {2 g^{-1} (p)}$. Consider on $\Omega$ the negative function $F(z) := \log |T(z)| + \log \sqrt {g (p)} - \log \sqrt{2}$. For $p \in \Omega$, we will verify properties (a) -- (c) in Definition  \ref{Green def}.

  \medskip
  
For (b), as $z \to p$,  $T(z) \to 0$ so $F(z) \to -\infty$. Moreover,  as $z \to p$,
\begin{align*}
F(z) - \log|z-p| &=   \log \left | \frac{T(z) - T(p)} {z-p} \right| + \log \sqrt {g (p)} - \log \sqrt{2} \\
&\to  \log |T^{\prime}(p)|   + \log \sqrt {g (p)} - \log \sqrt{2} \\
&=\log \sqrt {g (p)} - \log \sqrt{2}.
\end{align*}

 \medskip
 
For (c), take any Dirichlet  regular boundary point $\zeta \in \partial \Omega$, then it holds that 
$$
K(z, z) \to \infty,
$$
as $\Omega \ni z \to \zeta$ (see \cite[Proposition 6]{PZ}). By the definition of the  Bergman-Calabi diastasis,
we know that
$$
\Phi_{p}(z):= \log \frac{K(z, z)  K(p, p)}{ |K(z, p)|^2}  \geq  \log \frac{K(z, z)  K(p, p)}{  \mathcal C_1 ^2} \to \infty,
 $$
 as $\Omega \ni z \to \zeta$, because $  \left|       K(z, p)       \right| $ is  bounded from above by a finite constant $\mathcal C_1>0$ for any $z  \in \Omega$.
 On the other hand, the explicit formula we  obtained in  \cite {DWo} says that
\begin{equation} \label{Calabi}
\Phi_{p}(z)=  {-2} \log  \left (1 - \frac{ g(p)}{2} |T(z )|^2 \right ), \quad z \in \Omega,
\end{equation} 
 so
 $|T(z)|^2 \to {2}{ g^{-1}(p)}$ uniformly and thus $F(z) \to 0^-$ as $\Omega \ni z \to \zeta$. Since by Kellogg's Theorem,  the set of irregular boundary points of $\Omega$ is a polar set, we have verified 
 that $F(z) \to 0$ as $  z \to \zeta$ for   n.e. $\zeta \in \partial \Omega$.
 
  \medskip
 
To verify  (a), we {\bf claim} that  $T (z) \neq 0$ on $\Omega \setminus \{p\}$. In fact, if there is a point $s \in \Omega  $ with $T (s) = 0$. Then, $\Phi_{p}(s )  = 0$, which implies that
$$ \frac{K(s, s)  K(p, p)}{ |K(s, p)|^2}  =1.$$
But this is not possible unless $s=p$ in view of the Cauchy-Schwarz inequality. 
By this claim, it follows that $F (z)$ is harmonic on $\Omega \setminus \{p\}$. 

  \medskip

Therefore, we conclude that  $F (z)$  is Green's function on $\Omega$ with a pole $p$. By Lemma \ref{Green lem}, we know that $T (z) $    is biholomorphic from $\Omega$   to a disc $\mathbb D_r :=\{ \eta \in \mathbb C :    |\eta |^2   < {2 g^{-1} (p)}  \}$ possibly less a relatively closed  polar set $P$.
Moreover, the   transformation formula of the Bergman kernel under   biholomorphism yields 
\begin{equation}  \label{relation}
K(z, p)  = T^{\prime} (z) \cdot \overline{T^{\prime}(p) }\cdot K_{\mathbb D_r \setminus P}\left(T (z),   0\right) = \frac{T ^{\prime}(z) } { \pi r^2} = \frac{ g  (p)} { 2 \pi   } T ^{\prime}(z), \quad \forall z \in \Omega,
\end{equation}
where the  second equality holds  due to the explicit formula of the Bergman kernel on the disc and the fact that
 the polar set is negligible for $L^2$ holomorphic functions.

    \medskip

For Part 1,  $T^{-1}$,
the inverse map of 
$T  $ extends across the polar part $P$.
If
$
  \left|       K(z, p)       \right|
 $
is additionally bounded from below  by a finite constant $\mathcal C_2>0$ for any $z  \in \Omega$, then $T^{-1}$ will be  Lipschitz and thus  continuous up to   $\overline {\mathbb D_r} $.
  
  \medskip

For Part 2, by \eqref{relation} and the  local $C^1$-connectivity, we know that
   $
  \left|     T^\prime(z)    \right|  
 $
is  bounded from above by a finite constant $\mathcal C_1>0$ for any $z  \in \Omega$.
   Therefore, similar to the proof of Theorem \ref{with B},
      we conclude that $T$ is   continuous up to   $\overline {\Omega} $.

   \medskip
  
Part 3  then follows from both Parts 1 and 2.
  
       \end{proof}
       
                \medskip
        
     The  fact below is well-known.
        
         \begin{pro} \label{exterior arc}   
         Let        $q $ be  a boundary point of a bounded domain $\Omega $ in $\mathbb C $. If $q$  is the endpoint of a simple arc lying in the exterior of  $\Omega$, then   $q$ is Dirichlet  regular. 
       \end{pro}

       We obtain the following result  for
     domains  for which the above  converse holds  true.

        \begin{theorem}     Let   $\Omega \subset \mathbb C $ be a bounded domain whose Bergman metric  has  Gaussian curvature identically equal to $- 2$. If 
       any Dirichlet  regular boundary point $q \in \partial \Omega$ is the endpoint of a simple arc $\gamma$ lying in the exterior of  $\Omega$, then 
        the Bergman representative coordinate $T (z) $ relative to any   $p \in \Omega$  is biholomorphic from $\Omega$   to a disc possibly less a relatively closed  polar set.

\end{theorem}

 \begin{proof} 
      Similar to the  proof of Theorem \ref{1-dim}, we only need to verify Condition (c) at each regular boundary point $q$. By assumption, $q$ is the endpoint of a simple arc lying in the exterior of  $\Omega$.  
       Without loss of generosity, we may assume that the simple arc is a straight line  $L$ extending to infinity on the extended complex plane $\mathbb C_\infty = \mathbb C \cup \{\infty\} $.         By the Riemann mapping theorem, $\mathbb C_\infty \setminus L$ is biholomorphic to a disc as it is simply connected. 
      Since any biholomorphism  is an isometry  with respect to the  Carath\'eodory metric, it follows that $\mathbb C_\infty \setminus L$ is Carath\'eodory complete and  the Carath\'eodory distance of
      $z, p \in \mathbb C \setminus \gamma$  tends to $\infty $ as $z \to  q \in \partial (\mathbb C_\infty \setminus L)$. 
      
   \medskip
      
      By the decreasing property, 
      since $\Omega \subset   ( \mathbb C_\infty \setminus L)$,     the Carath\'eodory distance on $\Omega$ 
      dominates that on $\mathbb C_\infty \setminus L$. Consequently,  the Carath\'eodory distance of $z, p \in \Omega$  tends to $\infty $ as $z \to q \in \partial {\Omega}$. 
  On the other hand,  as $z \to q$,
  the  Bergman-Calabi diastasis $\Phi_{p}(z)$ tends to $\infty $
  whenever the Carath\'eodory distance of   $z, p \in \Omega$   tends to $\infty $ (see \cite{Bur, Look}).
  Therefore, due to the explicit formula \eqref{Calabi},  it holds that 
 $|T(z)|^2 \to {2}{ g^{-1}(p)}$ uniformly and thus $F(z) \to 0^-$ as $\Omega \ni z \to q$. The rest of the proof is the same as that of  Theorem \ref{1-dim}.

 \end{proof}

   A domain for which the  converse in Proposition  \ref{exterior arc}    fails to hold can be found in 
 Example \ref{counter}.

     \medskip
     \medskip

       Our Theorem \ref{1-dim} directly yields the following corollary, since   the boundary of a bounded $L^2$-domain of holomorphy  contains no pluripolar part, cf. \cite{PZ, I04}.

     \medskip

\begin{cor} \label{cor 1-dim} Let   $\Omega \subset \mathbb C $ be a bounded $L^2$-domain of holomorphy whose Bergman metric $g$ has its  Gaussian curvature identically equal to $- 2$.
If there exists some point   $p \in \Omega$ such that 
     $
  \left|       K(z, p)       \right|  
 $
is  bounded from above by a finite constant $\mathcal C_1>0$ for any $z  \in \Omega$, then the Bergman representative coordinate $T (z) $ relative to $p$  is biholomorphic from $\Omega$   to a disc.
Moreover, 

 \begin{enumerate}

\item [$1.$]    if
 $
  \left|       K(z, p)       \right|  
 $
is  bounded from below  by a finite constant $\mathcal C_2>0$ for any $z  \in \Omega$, then $T  $ has a biholomorphic inverse map that extends continuously up to the closure of  the disc;  in particular,  
 $\partial \Omega$ is locally connected;

\item [$2.$]   

if     $\Omega$ satisfies 
  local $C^1$-connectivity around any boundary point $q  \in \partial \Omega$,  then $T  $ extends   continuously up to   $\overline\Omega $;

\item [$3.$]   

if both Part $1$ and Part  $2$ hold,  then $T  $ extends   to a homeomorphism of the closures;
in particular,  
 $\partial \Omega$ is a Jordan curve.

 \end{enumerate}

\end{cor}

     \medskip

 Furthermore, our  Corollary \ref{simply}  gives sufficient conditions for the extension of the Riemann map to the closure  of a bounded, simply-connected planar domains $D$  in terms of its Bergman kernel.
Recall that  $D$  is biholomorphic to a disc by the  Riemann mapping theorem, cf. \cite{GKim}.
In fact, the Bergman representative coordinate $T (z)  $  relative to $p$  is   biholomorphic from $D$   to a disc $\mathbb D_r:= \{ w\in \mathbb C :     |w|^2    < 2g^{-1}(p) \}$  such that $T (p)=0$ and $T^{\prime}(p)=1$, where $g$ is the Bergman metric of $D$.  (See, for instance, \cite[Chap. VI]{Ber70} or \cite{Lu}.)

           \medskip

        \begin{proof} [{\bf Proof of Corollary \ref{simply}}]

By \eqref{relation}, we  know
\begin{equation}  \label{relation!}
K(z, p)    = \frac{ g  (p)} { 2 \pi   } T ^{\prime}(z), \quad \forall z \in D.
\end{equation} 
 
For Part 1, if $
  \left|       K(z, p)       \right|
 $
is  bounded from above by a finite constant $ \mathcal C_1>0$ for any $z  \in D$, then $T  $ will have its derivative bounded from above. 
Similar to the proof of Theorem \ref{with B}, the  local $C^1$-connectivity around any $q  \in \partial D$   will imply  that $T $ is   continuous up to   $\overline { D} $.
 
   \medskip

For Part 2, denote by $\tau: \mathbb D_r \to D $ the inverse biholomorphic map of $T$. 
If
$
  \left|       K(z, p)       \right|
 $
is  bounded from below  by a finite constant $ \mathcal  C_2>0$ for any $z  \in D$, then $\tau$ will have its derivative bounded from above. 
By  the inverse function theorem and \eqref{relation!},
$|\tau ^{\prime} (\xi )|$
is  bounded from above by the finite constant $ { g   (p)}  (2 \pi \mathcal C_2)^{-1} >0 $ for all $\xi  \in \mathbb D_r$. 
Therefore, we conclude that $\tau $ is  Lipschitz and thus  continuous up to   $\overline {\mathbb D_r} $.
Moreover, $\tau (\partial \mathbb D_r) \subset \partial D$. 
Since $\tau(\overline{\mathbb D_r})$ is closed, 
 $\tau$ maps $\overline {\mathbb D_r} $  onto $\overline {D} $ and  thus $\tau (\partial \mathbb D_r) = \partial D$, which  shows that $\partial D$ is     locally connected. 

  \medskip

 Part 3   follows from both Parts 1 and 2.

      \end{proof}

     \medskip

\noindent{}{\bf Remarks.}   
       
         \medskip
         
\begin{enumerate}[label=(\roman*)]    

\item The condition in Part 1 of Corollary \ref{simply} is  only {\it sufficient} but not necessary for the continuous extension to the closure. This can be seen from  Kerzman's example  in \cite {Ker}, which was given as 
\begin{equation} \label{Kerzman's example}
\mathcal D:= \{z \in \mathbb C: z=r e^{i\theta}, 0<r<1, \frac{\pi}{2} < \theta < 2 \pi\}.
\end{equation}  
   Here, the boundary  $\partial  \mathcal D$ is a Jordan curve  but  
 $
  \left|       K(z, p)       \right| 
 $
is  not  bounded from above by a finite constant  for all $z  \in \mathcal  D$.  

  \medskip

\item    {\bf  Condition $(B)$ is not a biholomorphic invariant}. First, it is  satisfied for the unit disc $\mathbb D$, whose Bergman kernel  is written as
$$
K(z, p) = \frac{1}{\pi}  \frac{1}{(1- z \bar p)^2} \quad \text{and} \quad \frac{\partial }{\partial z}K(z, p) =    \frac{ 2 \bar p}{\pi (1- z \bar p)^3},
$$
for an open set $U:= \mathbb D_{\frac{1}{2}}$, which is a disc around the origin  with radius $\frac{1}{2}$.
To see this, notice that for any $z \in \mathbb D$ and any $p\in U$, it follows that 
$$
\left|\frac{\partial }{\partial z}K(z, p) \right|    \leq     \frac{1}{\pi  \left| 1- z \bar p \right|^3}  =    |K(z, p)|   \frac{1 }{|1- z \bar p|}  <  |K(z, p)|  \frac{1 }{1-  |p|}  < 2 |K(z, p)|.
$$
However, Condition $(B)$ is {\it not} satisfied  for the examples of Forn\ae ss  \cite{For} and   Kerzman \cite {Ker}, where  $|K(z,  p)|$, for any fixed $p$, is unbounded as $z$ approaches  certain boundary point;
this is because 
  $ \mathcal  D$ in \eqref{Kerzman's example} satisfies
 local $C^1$-connectivity around any $q \in \partial  \mathcal  D$, and if Condition $(B)$  is satisfied we  would then conclude \eqref{upper bound K}, which is a contradiction.
 One can also check   that in Kerzman's example the determinant of the Jacobian of a Bergman representative coordinate is unbounded (see \cite{DWo}).

\end{enumerate}

Although a bounded, simply-connected  domain $D$ in $\mathbb C$ is a Lu Qi-Keng domain, namely for any  $p \in D$, the Bergman kernel $K(\cdot , p)$  has no zero set,  we have the following observation.

\medskip

 \begin{pro} \label{simply counter} 
     Let $D \subset \mathbb C$ be a bounded, simply-connected  domain whose boundary   is  not  locally connected. Then,   for any $p\in D$,  neither $  \left|       K(z, p)       \right|  
     $ nor $|T^{\prime} (z)|$ 
is   bounded from below by a finite constant $ C>0$ for all $z  \in D$, i.e., 
\begin{equation} \label{tend to zero}
\inf_{z \in D} |K(z, p)|  = 0 = \inf_{z \in D} |T^{\prime} (z)|.
 \end{equation}

        \end{pro}

        \begin{proof} 
  For the first part, assume the contrary that for some $p\in D$,
  $$
\inf_{z \in D} |K(z, p)|  > 0.
 $$
  Then, by  Part 2 of Corollary \ref{simply}, $\partial D$ is necessary locally connected, which is a contradiction. The same thing happens to $T^{\prime} (z)$
  in view of \eqref{relation}, and we have completed the proof.

      \end{proof}

\medskip

        \begin{example}  \label{counter}
An example of a  bounded,   simply-connected  planar domain   $D$ whose boundary  is  not    locally connected  is given  in  the  following Figure \ref{Figure 1}.  In this example,  the boundary $\partial D$ consists of the comb space, which is not locally connected.
 Also, $\partial D$ is not a  continuous path;
 here,  the boundary $\partial D$ being  a continuous path means that there exists a continuous function $\gamma: [0, 1] \to \mathbb C$ such that $\gamma ([0, 1]) = \partial D$ and $\gamma (0) = \gamma (1)$, cf. \cite[Chapter 14]{Conway}. 
 Consequently, 
 for any $p\in D$,  we know by Proposition \ref{simply counter} that \eqref{tend to zero} must holds true.

\end{example} 

 \medskip

 {\center
 
\tikzstyle{every node}=[circle, draw, fill=black!50,
                    inner sep=0pt, minimum width=4pt]
\begin{tikzpicture}[thick,scale=0.9]  
  \draw[very thick] (0,0)--(7,0)--(7, -5)--(0, -5)--cycle;
\draw[very thick] \foreach \x in {0, 0.2, ..., 4.5}
 {
  (0+  \x^3/16, -5)   -- (0+\x^3/16, -2) 
 };

\end{tikzpicture}

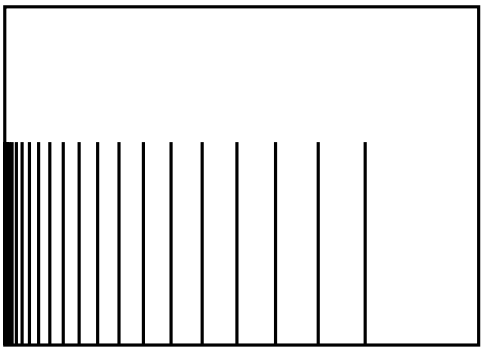
\captionof{figure}{A simply-connected domain with discontinuous boundary}  
 \label{Figure 1}
 
}

\bigskip
 
\bigskip

Lastly, we prove  the following estimate for planar domains whose Bergman metric  has constant negative Gaussian curvature.

\medskip

  \begin{pro} \label{1-dim lemma} 
     Let $\Omega \subset \mathbb C$ be a domain whose Bergman metric $
 g (z)dz\otimes d\overline{z}:= (\log K (z, z))_{ z\overline{z}}dz\otimes d\overline{z}
 $  has  Gaussian curvature identically equal to $-2$. 
  Then, for any $p \in \Omega$, the Bergman representative coordinate $T (z) $ relative to $p$ satisfies
$$
 |{T^\prime (z)}|  \leq  \frac{2 \pi }  { g (p)    } |K(z, p)| ,  \quad \forall z \in  \Omega.
$$
Moreover, equality holds on $\Omega$ if 
 $
  \left|       K(z, p)       \right|  
 $
is  bounded from above  by a finite positive constant for all $z  \in \Omega$.

  \end{pro}

      \begin{proof}

  By the discussion in Section \ref {2.1} (cf.  \cite{DWo, D}), the zero set   $A_{p}=\emptyset$ and   the  explicit formula \eqref{formula} of the Bergman-Calabi diastasis relative to $p$  is given by
 $$
\Phi_{p}( z) =   -2  \log      \left (1 - \frac{g (p) }{2}   |T(z)|^2     \right ), \quad \forall z \in  \Omega,
$$ 
which yields 
  $$
g(z)  =  | T^{\prime}(z)|^2 g (p) \left (1 - \frac{g (p) }{2}    |T(z)|^2     \right )^{-2}, \quad \forall  z \in  \Omega.
 $$ 
Therefore, by the above two formulas and \eqref{dia}, it holds that
 \begin{equation} \label{longineq}
 \frac{|K(z, p)|^2} {K(z, z)K(p, p)}  = e^{ -\Phi_p} (z)  = \left(1 - \frac{g (p) }{2}   |T (z)|^2   \right)^{2}   = \frac{ |{T^\prime (z)}|^2   g (p)}{ g (z)},  \quad \forall  z \in  \Omega.
\end{equation}
From \cite{D}, we know that for any $w\in \Omega$, 
 \begin{equation} \label{ineq}
\pi  K(w, w) \geq c_B^2(w) \geq      \frac{g(w)}{2},
\end{equation}
where $c_B$ is the analytic capacity (Carath\'eodory metric) defined as 
$$
 c_B(w):=\sup \left \{ | h^{\prime}(w) | \, :  \,  h\text{ is holomorphic on $\Omega$ with $h({w})= 0$ and} \, \, |h|\leq1 \right\}.
 $$
 Notice that the first inequality of \eqref{ineq} was proved by  Suita \cite{Su}.
Therefore,  by \eqref{longineq} and  \eqref{ineq}, we get
$$
 { |K(z, p)| }   = \sqrt{ \frac{   g (p) }{g (z) }  K(z, z)K(p, p)   }  |{T^\prime (z)}| \geq \frac{   g (p)    }{2 \pi  } |{T^\prime (z)}| ,  \quad \forall  z \in  \Omega.
$$
 
 By Theorem \ref{1-dim}, we know that if 
 $
  \left|       K(z, p)       \right|  
 $
is  bounded from above  by a finite positive constant for any $z  \in \Omega$, then $T$ becomes a biholomorphism and thus equality holds on $\Omega$ due to the transformation formula of the Bergman kernel. 

       \end{proof}

      Some  rigidity theorems related to capacities were given in \cite{DT, DTZ}.

\section{Bounded  Bergman Representative Coordinates}
In this section, we  study the boundedness of the 
Bergman representative coordinate. 
In fact, if there exist a point $p $  and its neighborhood $ U_p $  such that  \eqref{2 similar} holds true, then  the Bergman representative coordinate  relative to $p$ becomes  a bounded holomorphic map.  More precisely, we  get

\medskip
 
 \begin{pro}  \label{Bergman bounded}
 
 Let $\Omega$ be a   domain  admitting the Bergman metric in $  \mathbb C^n, n\geq 1$,  and let $K(z, \zeta)$ denote the Bergman kernel of $\Omega$. Assume that there exists a point $p \in \Omega$ 
 such that 
\begin{equation}  \label{similar}
 \sup_{\zeta \in U } \left|   K(z, \zeta)\right|  \le \mathcal C  \left| K(z, p)\right|, \quad  \forall   z \in \Omega, 
\end{equation}
 where  $U $  is a neighborhood of $p$ and 
 $\mathcal C >0$ is a finite constant.
  Then,  
the Bergman representative coordinate $T (z) $ relative to  $p$ maps $\Omega$ holomorphically to a bounded domain in $ \mathbb C^n$.  
 
\end{pro}

\begin{proof}  

We first assume that $\Omega$ satisfies   \eqref{similar}  at some   $p $  with \eqref{normal}.   Let $A_{p}$ be the zero set  of the Bergman kernel $K(\cdot , p)$. 
Then, 
 for each  fixed $ j = 1, \dotsc, n$,
$$
w_{j} (z):=   K^{-1}(z, p)  \left. \frac{\partial}{\partial \overline {t_j} } \right|_{t=p} K(z, t)  -    \left. \frac{\partial}{\partial \overline {t_j} } \right|_{t=p}  \log K(t, t), \quad z \in \Omega  \setminus A_{p}.
 $$
 Take a small polydisc $\mathbb D^n(p; r_p)  \subset U $  for some  $r_p>0 $. For each $ j = 1, \dotsc, n$, by Cauchy's integral formula for derivatives and \eqref{similar},   it holds that
 \begin{align*}
\left|\left. \frac{\partial}{\partial \overline {t_j} } \right|_{t=p} K(z, t) \right| 
& =   \left| \frac{1}{2\pi  i}  \int_{\{|t_j-p_j|= r_p \}}  \frac{ K((p_1, \dotsc,  t_j, \dotsc, p_n), z)  }{(t_j - p_j) ^{2}}   dt_j \right|\\
      &\leq  \frac{1}{2\pi {r_p  }}  \left|  \int_0^{2\pi}    K(  (p_1, \dotsc,  p_j + r_p e^{i \theta}, \dotsc, p_n),  z)      d \theta \right|\\
        & \leq   \frac{1}{  {r_p  }}  \sup_{  U}  |K(  \cdot ,  z)|  \\
        &  \leq   \frac{\mathcal C}{  {r_p  }}    |K(  p,  z)|,  \quad  \forall  z \in \Omega \setminus A_{p}.
\end{align*}
 Therefore,    for any   $z \in \Omega  \setminus A_{p}$,  
 $$
|w_j(z)|    \leq \left | K^{-1}(z, p)   \left. \frac{\partial}{\partial \overline {t_j} } \right|_{t=p} K(z, t)  -    \left. \frac{\partial}{\partial \overline {t_j} } \right|_{t=p}  \log K(t, t)\right |
   \leq        \frac{\mathcal C}{  {r_p  }}   + C_p.
 $$
By the Riemann removable singularity theorem, the bounded holomorphic function  $w_j$ extends across the analytic variety $A_p$ to the whole domain $\Omega$, which completes the proof.
 
\medskip

If  $\Omega$ satisfies    \eqref{similar}  at some general point   $p \in \Omega$, 
 let  
    $g_{\alpha \bar \beta} (p)$ be the positive-definite  Hermitian matrix associated with the Bergman metric at $p$.
One performs a   linear transformation $F$ from $\Omega$ to  $\Omega^1$ such that the Bergman metric $g^1$ on $\Omega_1 $  satisfies 
$
  g^1_{\alpha \bar \beta} (F (p)) =\delta_{\alpha  \beta }.
  $
 Since $F$ is a biholomorphism,  the Bergman kernels on $\Omega$ and $\Omega^1$  differ by a multiple constant,
  which is the determinant square of  $g_{\alpha \bar \beta} (p)$, due to the transformation rule. Therefore, $\Omega^1$ satisfies    \eqref{similar} with respect to the  point   $F (p) $. By the previous argument,   the Bergman representative coordinate $T^1 (z)  $   relative to $F (p)$  maps $\Omega^1$  holomorphically  to a bounded domain in $ \mathbb C^n$.    Finally, the composition map $F^{-1}\circ T^1 \circ F$ is the Bergman representative coordinate $T (z)  $ relative to $ p$ and it is 
  holomorphic  from $\Omega$  to a bounded domain in $ \mathbb C^n$.

 \end{proof} 

\medskip

Using Proposition \ref{Bergman bounded} and the authors'   results in \cite{DWo}, we will prove Theorem \ref{kernel similar}.
 
\medskip

\begin{proof}  [{\bf Proof of Theorem \ref{kernel similar}}]
Assume that the  holomorphic sectional curvature of the Bergman metric $g$ is identically  $-c^2$.
For any $p \in \Omega$, we take  a small enough convex neighborhood  $U_p $ such that   the set $T (U) $ is contained in  a ball  $\mathcal B$  defined in \eqref{ball}.
For any  $\zeta \in U_p$,
by the Cauchy-Schwarz inequality and mean value theorem, there exists a point $\eta = s\zeta + (1-s) p  \in U_p$, where $s\in (0, 1)$,
such that
 $$
|K(\zeta, t) - K(p, t)| \leq |\left. \nabla_{ {w}} \right|_{w= \eta}  K(w, t) | \cdot  |\zeta-p|,
$$
 where
 $\nabla_w   := (\frac{\partial  }{\partial   {w_1} } , ... ,  \frac{\partial  }{\partial   {w_n} } )$ denotes the complex gradient operator.
 
 \medskip
 
By \cite{DWo} and   \eqref{less} in the proof of Proposition \ref{upper bound},    it holds for any $\zeta \in U_p$ and any $ z\in \Omega$ that
\begin{align*}
 |K(\zeta, z)| & \leq  | K(p, z)| +  \sqrt{  \sum_{j=1}^n \left|\left. \frac{\partial}{\partial   {w_j} } \right|_{w=\eta} K(w, z) \right| ^2} \cdot  |\zeta-p| \\
& \leq  | K(p, z)| +  \sqrt{  n }  C_\eta       |K(  \eta,  z) |    \cdot  |\zeta-p|,
\end{align*}
where 
$$
C_\eta:=\sum_{j =1}^n  \left|    \left. \frac{\partial}{\partial   {z_j} } \right|_{z=\eta}  \log K(z, z) \right| +  n {\sqrt 2}{c^{-1}}   
$$
is   a finite positive constant depending on $\eta$. Thus,  $
C_{U_p} :=  \sup_{  \zeta \in  {U_p}} C_\zeta  >0
$
is also a finite constant as   the Bergman kernel is locally uniformly bounded.  Choosing a smaller neighborhood $U_1$ of $p $ such that  $ \sqrt{  n }  C_{U_p}     |\zeta-p| <\frac{1}{2}$ 
 whenever  $\zeta \in U_1$, we will get
 $$
 |K(\zeta, z)|   \leq  | K(p, z)| + \frac{1}{2}     |K( \eta,  z) | \leq | K(p, z)| + \frac{1}{2}  \sup_{U_1}   |K( \cdot,  z) |, \quad \forall  \zeta \in U_1, \, \forall z \in \Omega.
$$
Since the above right hand side is independent of $\zeta$, it follows that
 $$
\sup_{U_1}   |K(\cdot, z)| \leq 2 | K(p, z)|,  \quad \forall    z \in \Omega.
$$
 For simplicity, still denote $U_1$ by $U_p$, and we have completed the proof.

\end{proof}

\medskip

  We remark that \eqref {2 similar} indeed holds on symmetric bounded domains, cf. \cite[Proposition 2.2]{DLT}.  
  For general simply-connected complete  K\"{a}hler manifolds with sectional curvatures bounded between negative constants, one expects a good hold of the Bergman kernel as pointed out by Greene and Wu in \cite[Chap. 8]{GW}.
  
  \medskip

   \subsection*{Statements and Declarations}

No financial or non-financial interests that are directly or indirectly related to the work submitted for publication was reported by the authors. 

Data sharing not applicable to this article as no datasets were generated or analysed during the current study.

  \subsection*{Funding}
The research of the first author was partially supported by an AMS-Simons Travel Grant and NSF Grant DMS-2103608.

\subsection*{Acknowledgements}
{\fontsize{11.5}{10}\selectfont

The authors thank the referee for the thoughtful and detailed report. The authors are grateful to Song-Ying Li for his insightful comments on this paper. The first author is grateful to Xiaojun Huang for valuable conversations during the Northeast Workshop in Geometric Analysis (NEWGA) held at UConn in 2022.}

\bibliographystyle{alphaspecial}

\begin{thebibliography}{HD} 


{\fontsize{11}{11}\selectfont



  \bibitem {AHP} \textsc{B. Avelin, L. Hed and H. Persson}, \emph{A note on the hyperconvexity of pseudoconvex domains beyond Lipschitz regularity}, Potential Anal.   {\bf 43} (2015),  531--545. 



 \bibitem  {BBMV} \textsc{G. P. Balakumar, D. Borah,  P. Mahajan and K. Verma}, \emph{Remarks on the higher dimensional Suita conjecture}, Proc. Amer. Math. Soc.  {\bf 147} (2019), 3401--3411. 

 

  \bibitem {Bell81} \textsc{S. R. Bell}, \emph{Biholomorphic mappings and the $\bar \partial$-problem}, Ann. of Math.  {\bf 114} (1981), 103--113.

  \bibitem {Bell06} \textsc{S. R. Bell}, \emph{Bergman coordinates}, Studia Math.  {\bf 176} (2006), 69--83. 

 \bibitem {BB} \textsc{S. R. Bell and H. P. Boas}, \emph{Regularity of the Bergman projection in weakly pseudoconvex domains}, Math. Ann. {\bf 257} (1981), 23--30. 

 \bibitem {BK} \textsc{S. R. Bell and S. G. Krantz}, \emph{Smoothness to the boundary of conformal maps}, Rocky Mountain J. Math.  {\bf 17} (1987), 23--40.
 
 
  



 
 \bibitem {BLi} \textsc{S. Bell and E. Ligocka}, \emph{A simplification and extension of Fefferman's theorem on biholomorphic mappings}, Invent. Math. {\bf 57} (1980), 283--289.
 
 
 
 
 
  \bibitem {Berceanu} \textsc{S. Berceanu}, \emph{Bergman Representative Coordinates on the Siegel-Jacobi Disk},  
  Rom. J. of Phys.  {\bf 60} (2015), 867--896.
  
   

  
 
 \bibitem {Ber} \textsc{S. Bergmann}, \emph{{\"U}ber die {Existenz} von {Repr{\"a}sentantenbereichen} in der {Theorie} der {Abbildung} durch {Paare} von {Funktionen} zweier komplexen {Ver{\"a}nderlichen}}, Math. Ann.  {\bf 102} (1930), 430--446.
 
  \bibitem {Ber70} \textsc{S. Bergman}, \emph{The kernel function and conformal mapping}, 2nd ed.,  Math. Surv., 5. Amer. Math. Soc., Providence, R.I., 1970. 
    
     
    
    
 
 \bibitem {BL} \textsc{B. Berndtsson and L. Lempert}, \emph{A proof of the Ohsawa-Takegoshi theorem with sharp estimates}, J. Math. Soc. Japan {\bf 68} (2016), 1461--1472.
 
 
  \bibitem {BCOV} \textsc{M. Bershadsky, S. Cecotti, H. Ooguri and C. Vafa}, \emph{Kodaira-Spencer theory of gravity and exact results for quantum string amplitudes},  Comm. Math. Phys.  {\bf165} (1994), 311--427.


 
 \bibitem {Bl13} \textsc{Z. B\l{}ocki}, \emph{Suita conjecture and the Ohsawa-Takegoshi extension theorem}, Invent. Math. {\bf193} (2013), 149--158.
 
 \bibitem {Bl} \textsc{Z. B\l{}ocki}, \emph{Cauchy-Riemann meet Monge-Amp\`{e}re}, Bull. Math. Sci. {\bf 4} (2014), 433--480.
 
 
\bibitem {BZ15} \textsc{Z. B\l{}ocki and W. Zwonek}, \emph{Estimates for the Bergman kernel and the multidimensional Suita conjecture}, New York J. Math. {\bf 21} (2015), 151--161.




 
\bibitem{Bo} \textsc{H. P. Boas}, \emph{Lu Qi-Keng's problem}, Several complex variables (Seoul, 1998), J. Korean Math. Soc.  {\bf 37} (2000), 253--267. 
 
 
 \bibitem{Bre} \textsc{H. J.  Bremermann}, \emph{Holomorphic continuation of the kernel function and the Bergman metric in several complex variables}, Lectures on functions of a complex variable, 349--383, University of Michigan Press, Ann Arbor, Mich., 1955. 
 
 
  \bibitem{Bur} \textsc{J. Burbea}, \emph{A generalization of Pick's theorem and its applications to intrinsic metrics}, Ann. Polon. Math.  {\bf 39} (1981), 49-61.



 
  \bibitem{Cara} \textsc{C.  Carath{\'e}odory}, \emph{{\"U}ber die gegenseitige {Beziehung} der {R{\"a}nder} bei der konformen {Abbildung} des {Inneren} einer \emph{Jordan} schen {Kurve} auf einen {Kreis}} (German), {Math. Ann.}  {\bf 73} (1913),  305--320.
  
  
    
 
 

\bibitem{Chen} \textsc{B.-Y. Chen},  \emph{Every bounded pseudo-convex domain with H\"{o}lder boundary is hyperconvex}, Bull. London Math. Soc.   {\bf 53} (2021), 1009--1015.

 



\bibitem   {CW} \textsc{W.-S. Cheung and B. Wong}, \emph{Remarks on two theorems of Qi-Keng Lu}, Sci. China Ser. A {\bf 51} (2008), 773--776.   



\bibitem {Conway} \textsc{J. B. Conway}, \emph{Functions of One Complex Variable. II.}, Grad. Texts in Math., 159, Springer-Verlag, New York, 1995.
 
 
  
 
 \bibitem  {DF} \textsc{K. Diederich and J. E. Forn\ae ss}, \emph{Biholomorphic mappings between certain real analytic domains in $\mathbb C^2$}, Math. Ann.  {\bf 245} (1979), 255--272.  
 
 
 
 
\bibitem {Dinew} \textsc{\.Z. Dinew},  \emph{On the Bergman representative coordinates}, Sci. China Math.  {\bf 54} (2011), 1357--1374.  


 \bibitem {D} \textsc{R. X. Dong}, \emph{Equality in Suita's conjecture and metrics of constant Gaussian curvature}, \href{https://arxiv.org/abs/1807.05537}{{\color{blue} {arXiv:1807.05537}}}.


 \bibitem {DLT} \textsc{R. X. Dong, S.-Y. Li  and J. N. Treuer}, \emph{Sharp pointwise and uniform estimates for $\bar \partial$}, Anal. PDE.  {\bf 16}-2 (2023), 407--431.
 


 \bibitem {DT} \textsc{R. X. Dong and J. Treuer}, \emph{Rigidity theorem by the minimal point of the Bergman kernel}, J. Geom. Anal. {\bf 31} (2021),  4856--4864.
 

 \bibitem {DTZ} \textsc{R. X. Dong, J. N. Treuer and Y. Zhang}, \emph{Rigidity theorems by capacities and kernels}, Int. Math. Res. Not., IMRN {\bf 2023}-24 (2023), 21180--21214.
 

 
\bibitem {DWW} \textsc{R. X. Dong, R. Wang and B. Wong}, \emph{Local Rigidity of the Bergman Metric and of the K\"ahler Carath\'eodory Metric}, \href{https://arxiv.org/abs/2408.09572}{{\color{blue} {arXiv:2408.09572}}}.
 

 
 \bibitem{DWo} \textsc{R. X. Dong and B. Wong}, \emph{Bergman-Calabi diastasis and K\"ahler metric of constant holomorphic sectional curvature}, Pure Appl. Math. Q. (Special Issue in honor of Joseph J. Kohn) {\bf 18} (2022), 481--502.
 
 
 
\bibitem{Fef} \textsc{C. Fefferman}, \emph{The Bergman kernel and biholomorphic mappings of pseudoconvex domains}, Invent. Math. {\bf 26} (1974), 1--65.



\bibitem  {For} \textsc{J. E. Forn\ae ss}, \emph{Biholomorphic mappings between weakly pseudoconvex domains}, Pacific J. Math. {\bf 74} (1978), 63--65.
 
 
 

\bibitem  {FW} \textsc{S. Fu and B. Wong}, \emph{On strictly pseudoconvex domains with K\"{a}hler-Einstein Bergman metrics}, Math. Res. Lett. {\bf 4} (1997), 697--703. 



\bibitem  {GKim} \textsc{R. E. Greene and K.-T. Kim},  \emph{The Riemann mapping theorem from Riemann's viewpoint}. Complex Anal. Synerg.  {\bf 3} (2017): 1, 11 p.
 
 
\bibitem  {GKK} \textsc{R. E. Greene, K.-T. Kim and S. G. Krantz},  \emph{The geometry of complex domains}. Progr. Math., 291. Birkh{\" a}user Boston, Ltd., Boston, MA, 2011.  
 
 
 \bibitem  {GK} \textsc{R. E.  Greene and S. G. Krantz}, \emph{The automorphism groups of strongly pseudoconvex domains}, Math. Ann.  {\bf  261} (1982), 425--446. 



 
\bibitem  {GW} \textsc{R. E. Greene and H. Wu}, \emph{Function theory on manifolds which possess a pole}. Lecture Notes in Math., 699. Springer, Berlin, 1979.



\bibitem {GZa} \textsc{Q. Guan and X. Zhou}, \emph{Optimal constant in an $L^2$ extension problem and a proof of a conjecture of Ohsawa}, Sci. China Math. {\bf 58} (2015), 35--59.

 



\bibitem {GZ} \textsc{Q. Guan and X. Zhou}, \emph{A solution of an $L^2$ extension problem with optimal estimate and applications}, Ann. of Math. {\bf 181} (2015), 1139--1208.

\bibitem {Hen} \textsc{G. M. Henkin}, \emph{An analytic polyhedron is not holomorphically equivalent to a strictly pseudoconvex domain} (Russian), Dokl. Akad. Nauk SSSR {\bf 210} (1973), 1026--1029.



\bibitem {HX1} \textsc{X. Huang and M. Xiao}, \emph{A uniformization theorem for Stein spaces}, Complex Anal. Synerg. {\bf 6} (2020), article no. 6 (5 pp).
 
 

\bibitem {HX} \textsc{X. Huang and M. Xiao}, \emph{Bergman-Einstein metrics, a generalization of Kerner's theorem and Stein spaces with spherical boundaries}, J. Reine Angew. Math. {\bf 2021}-770 (2021), 183--203. 



 
 
\bibitem {I04} \textsc{M. A. S. Irgens}, \emph{Continuation of $L^2$-holomorphic functions}, Math. Z. {\bf 247} (2004), 611--617. 



\bibitem  {Ker} \textsc{N. Kerzman}, \emph{The Bergman Kernel Function. Differentiability at the Boundary}, Math. Ann. {\bf 195} (1972), 149--158. 

 
\bibitem  {Kra} \textsc{S. G.  Krantz}, \emph{Applications of Bergman representative coordinates}, Rocky Mountain J. Math. {\bf 50} (2020),  631--637. 


\bibitem  {Li} \textsc{S.-Y. Li}, \emph{Characterization for Balls by Potential Function of K\"{a}hler-Einstein Metrics for domains in $\mathbb C^n$}, Comm. Anal. Geom.  {\bf 13} (2005), 461--478.


\bibitem  {Lig} \textsc{E. Ligocka},  \emph{How to prove Fefferman's theorem without use of differential geometry}, Ann. Polon. Math. {\bf 39} (1981) , 117--130.


\bibitem {Look} \textsc{K. H. Look}, \emph{Schwarz lemma in the theory of functions of several complex variables}, Acta Math. Sinica {\bf 7} (1957), 370--420.
  
 
 

\bibitem {Lu} \textsc{Q.-K. Lu}, \emph{On K\"{a}hler manifolds with constant curvature} (Chinese), Acta Math. Sinica {\bf 16} (1966), 269--281; translation in Chinese Math.-Acta {\bf 8} (1966) 283--298. 

 

 
 \bibitem {NS} \textsc{S. Yu. Nemirovski and R. G. Shafikov}, \emph{Conjectures of Cheng and Ramadanov} (Russian), Uspekhi Mat. Nauk {\bf 370} (2006), 193--194; translation in Russian Math. Surveys {\bf 61} (2006), 780--782. 




 \bibitem {NWY} \textsc{L. Nirenberg, S. Webster and P. Yang}, \emph{Local boundary regularity of holomorphic mappings}, Comm. Pure Appl. Math.  {\bf 33} (1980), 305--338. 



\bibitem {O18} \textsc{T. Ohsawa}, \emph{$L^2$ approaches in several complex variables. Towards the Oka-Cartan theory with precise bounds}, 2nd ed., Springer Monogr. Math., Springer, Tokyo, 2018. 

 
\bibitem {O20} \textsc{T. Ohsawa}, \emph{A Survey on the $L^2$ Extension Theorems}, J. Geom. Anal. {\bf 30} (2020), 1366--1395.

\bibitem {O20B} \textsc{T. Ohsawa}, \emph{A Role of the $L^2$ Method in the Study of Analytic Families}, Bousfield classes and Ohkawa's theorem, 423--435, Springer Proc. Math. Stat., 309, Springer, Singapore, 2020.


 
 \bibitem {OT} \textsc{T. Ohsawa and K. Takegoshi}, \emph{On the extension of $L^2$ holomorphic functions}, Math. Z. {\bf 195} (1987), 197--204.

 
\bibitem {PZ} \textsc{P. Pflug and W. Zwonek}, \emph{$L^2_h$-domains of holomorphy and the Bergman kernel}, Studia Math. {\bf 151} (2002), 99--108. 



\bibitem{R95} T. Ransford, \emph{Potential theory in the complex plane}, London Math. Soc. Stud. Texts, 28. Cambridge Univ. Press, Cambridge, 1995.

 

\bibitem {S82} \textsc{J. Siciak}, \emph{On removable singularities of $L^2$ holomorphic functions of several complex variables}, 73--81, Prace Matematyczno--Fizyczne Wy{\.z}sza Szko\l{}a In{\.z}ynierska w Radomiu, 1982.  



 \bibitem {Su}  \textsc{N. Suita}, \emph{Capacities and kernels on Riemann surfaces}, Arch. Ration. Mech. Anal. {\bf 46} (1972), 212--217.

 \bibitem {Su73}  \textsc{N. Suita}, \emph{On a metric induced by analytic capacity}, {K{\=o}dai Math. Semin. Rep.}  {\bf 25} (1973), 215--218.

 

 \bibitem {Vo}  \textsc{N. Vormoor}, \emph{Topologische {Fortsetzung} biholomorpher {Funktionen} auf dem {Rande} bei beschr{\"a}nkten streng-pseudokonvexen {Gebieten} im {$\mathbb C ^n$} mit C{{\(^\infty\)}}-{Rand}} (German), Math. Ann. {\bf 204} (1973),  239--261.
 
   
 
 \bibitem {Web}  \textsc{S. M. Webster}, \emph{Biholomorphic mappings and the Bergman kernel off the diagonal}, Invent. Math.  {\bf 51} (1979), 155--169.
 
 
\bibitem {W} \textsc{B. Wong}, \emph{On the holomorphic curvature of some intrinsic metrics}, Proc. Amer. Math. Soc. {\bf 65} (1977), 57--61. 
 

  
\bibitem {Yoo} \textsc{S. Yoo}, \emph{A differential-geometric analysis of the Bergman representative map}, Ann. Polon. Math. {\bf 120} (2017), 163--181.

\bibitem{Z} \textsc{X. Zhou}, \emph{A Survey on $L^2$ extension problem}, {Complex Geometry and Dynamics}, 291--309, Abel Symp., 10, Springer, Cham, 2015.

}
\end{thebibliography}

\fontsize{11}{9}\selectfont

\vspace{0.5cm}

 \noindent xindong.math@outlook.com

\vspace{0.2 cm}

\noindent Department of Mathematics, University of Connecticut, Stamford, CT 06901-2315, USA

\vspace{0.4cm}

\noindent wong@math.ucr.edu,

\vspace{0.2 cm}

\noindent Department of Mathematics, University of California, Riverside, CA 92521-0429, USA

 \end{document}